\documentclass[10pt]{article}
\usepackage{amsmath}
\usepackage{amsfonts}
\usepackage{amssymb}
\usepackage{amsthm}
\usepackage{palatino}
\usepackage[margin=2.5cm, vmargin={1.5cm},includefoot]{geometry}

\title{Modular Dedekind symbols associated to Fuchsian groups and higher-order Eisenstein series}

\author{Jay Jorgenson\footnote{The first author was partially supported by NSF and PSC-CUNY grants.},
Cormac O'Sullivan\footnote{The second author was partially supported by a PSC-CUNY grant.} and Lejla Smajlovi\'c}

\begin{document}

\maketitle

\def\s#1#2{\langle \,#1 , #2 \,\rangle}

\def\H{{\mathbb H}}
\def\F{{\mathfrak F}}
\def\C{{\mathbb C}}
\def\R{{\mathbb R}}
\def\Z{{\mathbb Z}}
\def\Q{{\mathbb Q}}
\def\N{{\mathbb N}}
\def\st{{\mathbb S}}
\def\D{{\mathbb D}}
\def\B{{\mathbb B}}
\def\G{{\Gamma}}
\def\GH{{\G \backslash \H}}
\def\g{{\gamma}}
\def\L{{\Lambda}}
\def\ee{{\varepsilon}}
\def\K{{\mathcal K}}
\def\Re{\text{\rm Re}}
\def\Im{\text{\rm Im}}
\def\SL{\text{\rm SL}}
\def\GL{\text{\rm GL}}

\def\slz{\text{\rm SL}(2,\Z)}
\def\slr{\text{\rm SL}(2,\R)}

\def\sgn{\text{\rm sgn}}
\def\tr{\text{\rm tr}}
\def\F{\mathcal{F}}
\def\ca{{\mathfrak a}}
\def\cb{{\mathfrak b}}
\def\cc{{\mathfrak c}}
\def\cd{{\mathfrak d}}
\def\ci{{\infty}}

\def\sa{{\sigma_\mathfrak a}}
\def\sb{{\sigma_\mathfrak b}}
\def\sc{{\sigma_\mathfrak c}}
\def\sd{{\sigma_\mathfrak d}}
\def\si{{\sigma_\infty}}

\def\se{{\sigma_\eta}}
\def\sz{{\sigma_{z_0}}}

\def\sai{{\sigma^{-1}_\mathfrak a}}
\def\sbi{{\sigma^{-1}_\mathfrak b}}
\def\sci{{\sigma^{-1}_\mathfrak c}}
\def\sdi{{\sigma^{-1}_\mathfrak d}}
\def\sii{{\sigma^{-1}_\infty}}
\def\PSL{\text{\rm PSL}}
\def\vol{\text{\rm vol}}
\def\I{\text{\rm Im}}

\newcommand{\m}[4]{\begin{pmatrix}#1&#2\\#3&#4\end{pmatrix}}
\newcommand{\n}{\frac1{\sqrt{37}}}
\newcommand{\ms}[4]{\left(\smallmatrix #1&#2\\#3&#4\endsmallmatrix\right)}
\newcommand{\ns}{\textstyle\frac1{\sqrt{37}}}

\newtheorem{theorem}{Theorem}[section]
\newtheorem{lemma}[theorem]{Lemma}
\newtheorem{prop}[theorem]{Proposition}
\newtheorem{cor}[theorem]{Corollary}
\newtheorem{conj}[theorem]{Conjecture}
\newtheorem{defs}[theorem]{Definitions}
\newtheorem{remark}[theorem]{Remark}
\renewcommand{\labelenumi}{(\roman{enumi})}

\numberwithin{equation}{section}
\newtheorem{main-theorem}{Theorem}
\newtheorem{main-prop}[main-theorem]{Proposition}

\bibliographystyle{plain}

\begin{abstract}\noindent
Let $E(z,s)$ be the  non-holomorphic Eisenstein series for the modular group $\slz$. The classical Kronecker limit formula shows that the second term in the Laurent expansion at $s=1$ of $E(z,s)$ is essentially the logarithm of the Dedekind eta function. This eta function is a weight $1/2$ modular form and Dedekind expressed its multiplier system in terms of Dedekind sums. Building on work of Goldstein, we extend these results from the modular group to more general Fuchsian groups $\G$. The analogue of the eta function  has a multiplier system that may be expressed in terms of a map $S:\G \to \R$ which we call a {\em modular Dedekind symbol}. We obtain detailed properties of these symbols  by means of the limit formula.

Twisting the usual  Eisenstein series with powers of additive homomorphisms from $\G$ to $\C$ produces higher-order Eisenstein series. These series share many of the properties of $E(z,s)$ though they have a more complicated automorphy condition. They  satisfy a Kronecker limit formula and produce higher-order Dedekind symbols $S^*:\G \to \R$. As an application of our general results, we
prove that higher-order Dedekind symbols associated to genus one congruence groups $\G_0(N)$ are rational.
\end{abstract}

\section{Introduction}

\subsection{Kronecker limit functions and Dedekind sums}
Write elements of the upper half plane $\H$ as $z=x+iy$ with $y>0$.
The non-holomorphic
Eisenstein series $E_{\infty}(s,z)$, associated to the full modular group $\SL(2, \Z)$, may be given as
$$
E_{\infty}(z,s) := \frac{1}{2} \sum_{(c,d)\in\Z^2, (c,d)=1} \frac{y^s}{|cz+d|^{2s}}.
$$
The classical {\em Kronecker limit formula} is the evaluation of the first two terms in the Laurent series expansion at $s=1$ of this Eisenstein series:
\begin{equation}\label{KLT}
\frac{\pi}3 E_{\infty}(z,s)= \frac{1}{s-1} +2 - 24\zeta'(-1)-2\log (4\pi) - \log \left(y |\eta(z)|^4 \right)
+  O(s-1)
\end{equation}
as $s\to 1$, where $\zeta(s)$ is the Riemann zeta function. The Dedekind eta function is given by
\begin{equation}\label{etad}
\eta(z):= q_z^{1/24} \prod_{n=1}^{\infty}(1-q_z^n)
\end{equation}
with $q_z=e(z) :=\exp(2\pi i z)$.  The fact that $E_{\infty}\bigl(\frac{az+b}{cz+d},s\bigr) = E_{\infty}(z,s)$ for all $(\smallmatrix a
& b \\ c & d \endsmallmatrix ) \in \SL(2, \Z)$ can be used to show  that 
one has
the relation
\begin{equation}\label{log_eta}
\log \eta\left(\frac{az+b}{cz+d}\right) = \log \eta(z) + \frac{1}{2}\log(cz+d) + \pi i S(a,b,c,d)
\end{equation}
for certain numbers $S(a,b,c,d)$, independent of $z$. Dedekind succeeded in finding these numbers explicitly and  proved
that
\begin{equation}\label{Dsum}
S(a,b,c,d)= \frac{a+d}{12c}-\frac{c}{|c|}\left( \frac{1}{4} + s(d,|c|)\right) \qquad \text{for} \qquad c\neq 0
\end{equation}
 and
 \begin{equation}\label{Dsum2}
S(a,b,0,d)= \frac{b}{12d}+\frac{d/|d|-1}{4}.
\end{equation}
The term $s(d,|c|)$ in \eqref{Dsum} is a {\em Dedekind sum} and defined as
\begin{equation}\label{dede}
    s(h,k):= \sum_{m=0}^{k-1} \biggl(\biggl( \frac{hm}{k}\biggr)\biggr) \biggl(\biggl(  \frac{m}{k}\biggr)\biggr) \quad \text{for} \quad
    \left(\left( x\right)\right) := \begin{cases} x - \lfloor x\rfloor - 1/2 & \text{if} \quad x \in \R, x\not\in \Z\\
    0 & \text{if} \quad x \in \Z
    \end{cases}
\end{equation}
where $h$ and $k$ are relatively prime integers  and $k\geq 1$. In \eqref{log_eta}, $ \log \eta(z)$ refers to the branch of $\eta(z)$ given explicitly in \eqref{logetadef}. Also $\log(cz+d)$ means  the principal branch of the logarithm with argument in $(-\pi,\pi]$.

Chapter 6 of the beautifully written monograph \cite{RG72}, appropriately entitled \it Some remarks on the history of the Dedekind sums, \rm
provides an elegant account of the discovery and subsequent applications and manifestations of Dedekind sums, which includes
pseudo-random number generators \cite{Kn77} and aspects of the Atiyah-Bott-Singer index theory \cite{Hi73} and \cite{Za72}.  More recently,
the article \cite{At87} describes further studies in both mathematics and physics which include Dedekind sums.

There are many
results which take (\ref{dede}) as a starting point and re-interpret the series in various ways, such as specializations of Bernoulli
polynomials, as cotangent sums, and in other manners. 
The purpose of this article is to develop a generalization of Kronecker's limit formula in the setting of general Fuchsian groups and higher-order
non-holomorphic Eisenstein series.  Once this result is established, we obtain a generalization of the Dedekind eta function
(\ref{log_eta}) from which we define the associated modular Dedekind symbols.  Prior to stating our main theorems, it is necessary to
set up our notation and recall known results.

\subsection{A Kronecker limit formula for higher-order Eisenstein series} \label{scal}

 Let $\Gamma$ be any Fuchsian group of the first kind which acts on
the hyperbolic upper half-space $\H$ such that the quotient
$\Gamma \backslash \H$ has finite volume, which we denote by $V_{\G}$.
Furthermore, we assume that $\G \backslash \H$ has at least one cusp.
Following the notation from \cite{Iw2}, \cite{Iw},  let us fix representatives
for the finite number of $\G$-inequivalent cusps, label them
$\ca,\cb, \dots$, and use the scaling matrices $\sa,\sb, \dots$ to
give local coordinates near these cusps. Define the subgroup
$\G_\ca$ to be those  elements of $\G$ which fix  $\ca$. Then  $\sa \in \SL (2,\R)$ satisfies $\sa \infty = \ca$ and
$$
 \sa^{-1}\G_\ca \sa  = \left\{\left. \pm \left(\smallmatrix 1
& m \\ 0 & 1 \endsmallmatrix\right) \ \right| \  m\in \Z\right\}.
$$
The matrix $\sa$ is unique up to multiplication on the right by $\pm\left(\smallmatrix 1
& t \\ 0 & 1 \endsmallmatrix\right)$ for any  $t\in \R$.
As is conventional, by conjugating $\G$ if necessary we may assume
that one cusp is at infinity,  denoted  $\infty$, and its
scaling matrix $\sigma_{\infty}$ is the identity matrix $I$.

Let $S_2(\G)$ be the space of holomorphic
cusp forms of weight 2 with respect to $\G$.
For $f\in S_2(\G)$
and $m,n \geq 0$, we define, following \cite{Go99a},
\cite{Go99b}, and \cite{OS00}, the {\em higher-order non-holomorphic
Eisenstein series} associated to the form $f$ by the series
\begin{equation}\label{E m,n}
E^{m,n}_\ca(z,s;f):= \sum_{\g \in \G_\ca\backslash\G} \langle\g,f\rangle^{m}\overline{\langle\g,f\rangle}^{n} \Im(\sai\g z)^s,
\end{equation}
where the {\em modular symbol} is defined by
\begin{equation}\label{mod}
    \s{\g}{f} := 2\pi i\int_{z_0}^{\g z_0} f(w) \, dw
\end{equation}
and independent of $z_0 \in \H$. In this context the term {\em higher-order} refers to the automorphic properties of the series as described in section \ref{alg-int}. If $m+n >0$ then clearly $E^{m,n}_\ca(z,s;f)$ will always be $0$ if $f \equiv 0$, so we usually assume that $f \not\equiv 0$.
Throughout this article, we will consider the form
$f$ to be fixed, hence we will
abbreviate the notation and just write $E_\ca^{m,n}(z,s)$.

It has been shown that the series (\ref{E m,n}) converges
 absolutely for $\Re (s) > 1$, $z \in \H$ and admits a
meromorphic continuation to all $s$ in $\C$; see \cite{JO'S08}, \cite{Pe} and \cite{PR}.  The continuation of (\ref{E m,n}) established in
\cite{Pe} and \cite{PR} uses perturbation theory. The proof
from \cite{JO'S08}
continues the extension of Selberg's method \cite{Sel65} to
higher-order forms, see also \cite{JO'S05} and \cite{OS00},
and has the advantage of yielding strong bounds on both the
Fourier coefficients of $E_\ca^{m,n}(z,s)$ and its growth in $z$.
In addition, it is shown in \cite[Thm. 2.18]{PR} and in \cite[Thm. 25]{Ri} that in the case $m=n \geq 0$ the series \eqref{E m,n}
has a pole at $s=1$ of order $m+1$ and
\begin{equation}\label{emnpl}
\frac 1{A_m} E^{m,m}_\ca(z,s)=\frac{1}{(s-1)^{m+1}}+\frac{B^{(m)}_\ca(z)}{(s-1)^m}+
\frac{C_\ca^{(m)}(z)}{(s-1)^{m-1}}+O\left(\frac{1}{(s-1)^{m-2}}\right),
\end{equation}
as $s\to 1$ where
\begin{equation}\label{amdef}
A_m:=\frac{(16\pi^2)^m}{V_{\G}^{m+1}}m!^2 \|f\|^{2m}
\end{equation}
and
$\|f\|$ is the Petersson norm. As we will see, for any $m \geq 0$,
the Kronecker limit function $B^{(m)}_\ca(z)$ may be written as
\begin{equation*}
B_\ca^{(m)}( z)=-\log y+ b_{\ca}^{(m)}(0)+2 \Re (H_{\ca}^{(m)}(z)),
\end{equation*}
with
\begin{equation}\label{hamz}
H_{\ca}^{(m)}(z) := -\delta_{0m}\delta_{\ca\ci}V_\G \cdot iz/2+\sum_{k> 0} b_{\ca}^{(m)}(k) e(kz).
\end{equation}
 We are using the Kronecker delta notation where $\delta_{0m}$ and $\delta_{\ca\ci}$ are zero unless $m=0$ and $\ca=\ci$, respectively, when they equal $1$.

\subsection{Main results}
For $m=0$, similarly to Goldstein in \cite{Gn1}, we set
\begin{equation*}
  \eta_{\G,\ca}(z) = \eta_{\ca}(z):=\exp\left( -\frac{1}{2}H^{(0)}_{\ca}(z) \right),
\end{equation*}
obtaining a holomorphic function that transforms with a multiplier system of weight $1/2$ for $\G$. To describe the multiplier, write
\begin{equation*}
\log\eta_{\ca}(\g z)= \log\eta_{\ca}(z) + \frac{1}{2} \log j(\g, z) +  \pi i S_{\G,\ca}(\g),
\end{equation*}
for all $\g \in \G$ with $j((\smallmatrix a
& b \\ c & d \endsmallmatrix ),z):=cz+d$ and $ S_{\G,\ca}:\G \to \R$. Here, $\log\eta_{\ca}(z)$ means $-H^{(0)}_{\ca}(z)/2$ and $\log j(\g, z)$ means the logarithm's principal branch. Comparing the above with \eqref{log_eta}, we obtain an analogue, $\eta_{\G, \ca}(z)$, of the  classical Dedekind eta function and  $S_{\G,\ca}(\g)$ is an analogue \eqref{Dsum}. Goldstein called $S_{\G,\ca}(\g)$ a Dedekind sum associated to $\G$ but we shall describe it as a   {\em modular Dedekind symbol} to emphasize that it is associated to the modular forms $E_\ca(z,s)$ and $\eta_{\ca}(z)$.  In the literature, the closely related  {\em Dedekind symbol} usually refers to the Dedekind sum $s(h,k)$ and its generalizations.

When $m \geq 1$, one of our first results is to prove that
$H_{\ca}^{(m)}(z)$ is independent of $m$;
therefore, we will write
\begin{equation}\label{h_ind_m}
H^*_{\ca}(z):= H_{\ca}^{(m)}(z) \qquad \text{for} \qquad m \geq 1.
\end{equation}

For $m \geq 1$, the function $E^{m,m}_\ca(z,s)$ is not $\Gamma$-invariant; hence
neither $B^{(m)}_\ca(z)$ nor $H^*_{\ca}(z)$ are $\G$-invariant.
However, the transformation property
of $E^{m,m}_\ca(z,s)$ is easily computed; in fact, as shown in \cite{JO'S08}, one can
view $E^{m,m}_\ca(z,s)$ as a component of a section of a unipotent bundle defined on $\G \backslash \H$.
Let
\begin{equation}\label{ffff}
F_\ca(z,f)=F_\ca(z):=2\pi
i\int_{\ca}^{z}f(w)\,dw
\end{equation}
and
\begin{equation} \label{V f def}
    V_f := \frac{V_\G}{16 \pi^2 \| f\|^2} = \frac{m^2 A_{m-1}}{A_m} \quad \text{for} \quad m \geq 1.
\end{equation}
We prove that  the function
$$
\eta^*_{\ca}(z,f) =\eta^*_{\ca}(z):=\exp\left( -\frac{1}{2}H^*_{\ca}(z)+\frac{V_f}{4}|F_\ca( z)|^2\right)
$$
is a modular form of weight $1/2$ with respect to $\G$ up to a factor of modulus one;
from the proof of Proposition \ref{prop:properties of Hab} it is easy to deduce that $\eta^*_{\ca}(z)$ satisfies the transformation formula
\begin{equation} \label{eta transf formula}
\log\eta^*_{\ca}(\g z)= \log\eta^*_{\ca}(z) + \frac{1}{2} \log j(\g, z) +  \pi i S^*_{\ca}(\g).
\end{equation}
We call the function $S^*_{\ca}:\G \to \R$  a  {\em higher-order modular Dedekind symbol} and, as the notation suggests, the function $\eta^*_{\ca}(z)$ is also a generalization of the Dedekind eta function
(\ref{log_eta}).  Since the form $f \in S_2(\G)$ is usually fixed, we suppress it in the notation $\eta^*_{\ca}(z)$ and $S^*_{\ca}$. The form $f$ naturally reappears in their transformation properties though. Our main theorems describe how $S_{\ca}(\g)$ and $S^*_{\ca}(\g)$ change with different group elements $\g$ and cusps $\ca$. This requires two additional pieces of notation.

The Eisenstein series of weight $2$ for $\G$ may be defined using Hecke's method:
\begin{equation} \label{heck}
  E_{\ca,2}(z):= \lim_{s\to 0^+}  \sum_{\g \in \G_\ca\backslash\G} j(\sai\g, z)^{-2} \Im(\sai\g z)^s.
\end{equation}
Though $ E_{\ca,2}(z)$ does have weight $2$, it only becomes holomorphic if $1/ (V_\G \cdot y)$ is added. For any two cusps, $\ca$ and $\cb$, the difference $E_{\ca,2}(z) - E_{\cb,2}(z)$ is a holomorphic modular form of weight $2$ for $\G$. See Proposition \ref{ea2} for proofs of these statements.

Secondly, for any two matrices $M$, $N$ in $\SL (2,\R)$ define the {\em phase factor}
 \begin{equation}\label{w1}
\omega(M,N):=\bigl(-\log j(MN, z) + \log j(M, N z)+ \log j(N, z)\bigr)/(2\pi i).
\end{equation}
The right side of (\ref{w1}) is independent of $z \in \H$ and takes only the values $\{-1,0,1\}$. We will see how to explicitly evaluate $\omega(M,N)$ in terms of the signs of the bottom row entries of $M$, $N$ and $MN$.

\begin{main-theorem} \label{main1}
With the notation as above, we have the following identities
 for modular Dedekind symbols  associated to the Fuchsian group $\G$.
\begin{enumerate}
\item
For all $\g$ and $\tau$ in $\G$ we have
\begin{equation}\label{mn-s1}
S_\ca(\g \tau)=S_\ca(\g)+S_\ca(\tau) + \omega(\g , \tau).
\end{equation}
\item
For any pair of cusps $\ca$, $\cb$ and $\gamma \in \G$, we have that
$$
S_\ca(\g)=S_\cb(\g)+\frac{V_\G}{8\pi^2 i}  \langle \g,E_{\ca,2} - E_{\cb,2} \rangle.
$$
\end{enumerate}
\end{main-theorem}

For $\G=\slz$, equation \eqref{mn-s1} is given by Asai in \cite[Thm. 3]{Asai}.
We give the analogous results for the higher-order modular Dedekind symbols  associated to  $\G$ and $f\in S_2(\G)$ next.

\begin{main-theorem} \label{main2}

\begin{enumerate}
\item
For all $\g$ and $\tau$ in $\G$ we have
\begin{equation}\label{mn-s2}
S^*_\ca (\g \tau) =S^*_\ca(\g) +S^*_\ca(\tau) + \frac{V_f}{2\pi} \Im \left(\langle\g,f\rangle \overline{\langle\tau,f\rangle}\right)  + \omega(\g,\tau).
\end{equation}
\item
For any pair of cusps $\ca$, $\cb$ and $\gamma \in \G$, we have that
$$
S^*_\ca (\g)   = S^*_\cb (\g) -\frac{V_f}{\pi} \Im \left(\overline{F_{\ca}( \cb)} \cdot \langle \g,f \rangle \right).
$$
\end{enumerate}
\end{main-theorem}


Before continuing, let us point out an additional interpretation of Theorem 2.  Consider the case when
$\G = \G_{0}(p)$, which under its standard action on $\H$ has quotient space with two cusps, one at $\infty$
and one at $0$.  If we let $\ca = \infty$ and $\cb = 0$, then the term $F_{\ca}( \cb)$ in part (ii) of
Theorem 2 is minus the integral of the holomorphic weight two form $f$ from $0$ to $\infty$, which is simply
$L(1,f)$, the special value of the $L$ function associated to $f$ at $s=1$.  If the coefficients of $f$
are real, then by re-arranging the identity in part (ii), one obtains a formula for $L(1,f)$ in terms of
periods of $f$ and modular Dedekind symbols, namely for any $\gamma \in \Gamma$ we have the identity
$$
L(1,f) = \frac{\pi(S^{*}_{0}(\gamma)- S^{*}_{\infty}(\gamma))}{V_{f}\Im \left(\langle \g,f \rangle \right)}
= \frac{16\pi^{3}\| f\|^2(S^{*}_{0}(\gamma)- S^{*}_{\infty}(\gamma))}{V_{\G}\Im \left(\langle \g,f \rangle \right)}.
$$
This observation is particularly interesting in that it does not
depend on $\G$ being equal to $\G_{0}(p)$, but rather can be carried out for any group $\G$ which admits
two cusps when one cusp is at $\infty$ and the other at $0$.

In other words, as long as the underlying group $\G$ has two cusps, then the special value at $s=1$ of
the (formal) $L$-function associated to $f$ can be expressed in terms of periods of $f$ and values of
the modular Dedekind symbol.  Whereas such considerations motivate many studies throughout modern number
theory when $\G$ is arithmetic, we are not aware of similar investigations, or results, when $\G$ is
not arithmetic, other than part (ii) of Theorem 2.

An interesting consequence of the transformation properties in the first parts of Theorems \ref{main1} and \ref{main2}
is the following.

\begin{main-prop} \label{main3}
For all $\g_1$, $\g_2$ and $\g_3$ in $\G$ and any cusp $\ca$ we have the following relations. First
\begin{gather}
S^*_\ca (\g_1 \g_2 \g_3 ) -S^*_\ca (\g_1 \g_2  )-S^*_\ca (\g_1  \g_3 )-S^*_\ca ( \g_2 \g_3 )+S^*_\ca (\g_1)+S^*_\ca (\g_2)
+S^*_\ca (\g_3) \notag\\  =
 \omega(\g_1 \g_2,\g_3) -\omega(\g_1,\g_3)-\omega(\g_2,\g_3), \label{mn-s3}
 \end{gather}
which is a consequence of \eqref{mn-s2}.
Since the phase factors are integers, \eqref{mn-s1} and \eqref{mn-s3} imply that
\begin{align*}
    S_\ca(\g_1 \g_2) & \equiv S_\ca(\g_1)+S_\ca(\g_2), \\
    S^*_\ca (\g_1 \g_2 \g_3 ) & \equiv S^*_\ca (\g_1 \g_2  )+S^*_\ca (\g_1  \g_3 )+S^*_\ca ( \g_2 \g_3 )-S^*_\ca (\g_1)-S^*_\ca (\g_2)
-S^*_\ca (\g_3) 
\end{align*}
where $\equiv$ denotes equality in $\R/\Z$. Thus $S_\ca:\G\to \R/\Z$ is a homomorphism. If we let $\theta_\ca$ be
the difference $S^*_\ca-S_\ca$ then
\begin{equation*}
    \theta_\ca (\g_1 \g_2 \g_3 ) = \theta_\ca (\g_1 \g_2  )+\theta_\ca (\g_1  \g_3 )+\theta_\ca ( \g_2 \g_3 )-\theta_\ca (\g_1)-\theta_\ca (\g_2)
-\theta_\ca (\g_3).
\end{equation*}
\end{main-prop}

In section \ref{alg-int}, $S^*_\ca$ and $\theta_\ca$  are interpreted as third-order maps.
Returning to $\eta^*_{\ca}(z)$, we have seen with \eqref{eta transf formula} that
\begin{equation*} \label{ntrs}
    \eta^*_{\ca}(\g z) = j(\g,z)^{1/2} e^{\pi i S^*_{\ca}(\g)} \eta^*_{\ca}(z) \quad \text{for all} \quad \g \in \G.
\end{equation*}
It is clear from Theorem \ref{main2} and Proposition \ref{main3} that $e^{\pi i S^*_{\ca}(\g)}$ is not a multiplier system in the usual sense and satisfies more complicated higher-order relations.

We next consider the values taken by $S_\ca$ and $S^*_\ca$. Dedekind showed that for $\G=\slz$, $12 S_\ci(\g)$ is always an integer. As we see in section \ref{vass}, Vassileva in \cite{Va} demonstrates that the modular Dedekind symbols associated to the Hecke congruence groups $\G_0(N)$ and the cusp $\ci$ are always rational. Takada indicates a similar result for the principal congruence groups $\G(N)$ in \cite[p. 409]{Tak}. It is also shown in \cite[Thm. 16]{JST16} that $48S_\ci(\g)$ is always an integer for the  modular Dedekind symbols associated to the group $\G_0(N)^+$ for $N$ square-free. We review this last result and obtain formulas for these symbols in section \ref{helling}.

The above rationality results are extended to the higher-order modular Dedekind symbols $S^*_\ca$ associated to certain  genus one congruence groups in the following theorem.

\begin{main-theorem}  \label{thm: rationality}
Let $N \in\{11, 14, 15, 17, 19, 20, 21, 24, 27, 32, 36, 49\}$. Then $S_\ca^*(\g) \in\Q$ for all $\g\in\G_0(N)$ and all cusps $\ca$ of $\G_0(N)$ equivalent to  $1/v$ where $1 \leq v\mid N$ and $(v,N/v)=1$.
\end{main-theorem}

We study the examples $\G_0(11)$ and $\G_0(37)^+$  in particular detail. For the higher-order modular Dedekind symbols associated to the group $\G_0(11)$ we prove:

\begin{main-theorem} \label{main4}
For all $\g \in \G_0(11)$, the numbers $10S^*_\ci(\g)$ and $ 10S^*_0(\g)$ are always integers.
\end{main-theorem}

Therefore, for the examples of Theorem \ref{main4}, $e^{\pi i S^*_{\ca}(\g)}$ is always a $20$th root of unity and, for the associated eta function, $ \eta^*_{\ca}(z)^{20}$ is a real-analytic modular form of weight $10$ for $\G_0(11)$.

We are led to a natural question: for which groups $\G$ are the values of the modular Dedekind symbols, of either type, always rational? We will see in Corollary \ref{par_ell_rat} that for any $\G$ these symbols are necessarily rational on parabolic and elliptic group elements.

\subsection{An algebraic interpretation} \label{alg-int}

Following the description in \cite[Sect. 3]{JO'S08}, we may define a sequence  $\mathcal A^n(\G)$ of sets of smooth functions from $\H \to \C$ recursively as follows. Let $\mathcal A^0(\G):=\{\H \to 0\}$ and  for  $n\geq 1$ set
$$
\mathcal A^n(\G):=\Bigl\{\psi \, \Big| \, \psi(\g z) -\psi(z) \in \mathcal A^{n-1}(\G) \text{ \ for all \ } \g \in \G \Bigr\}.
$$
Elements of $\mathcal A^n(\G)$ are called {\em $n$th-order automorphic forms}. The classical $\G$-invariant functions, such as $E_\ca(z,s)$, are in
$\mathcal A^1(\G)$ and so are first-order. The series $E^{m,n}_\ca(z,s)$ is in $\mathcal A^{m+n+1}(\G)$ and we call it a higher-order automorphic form if $m+n \geq 1$.
If we let $\g \in \G$ act on $\psi$ by $(\psi|\g)(z):=\psi(\g z)$, and extend this action to all $\C[\G]$ by linearity, then we see that $\psi \in \mathcal A^{n}(\G)$ if and only if
\begin{equation*}
    \psi \big|(\g_1 -I)(\g_2-I) \cdots (\g_n-I) = 0  \quad \text{for all} \quad \g_1,\g_2, \dots,\g_n \in \G.
\end{equation*}

Similarly, as in \cite[Sect. 10]{IO}, one can define a related sequence $\operatorname{Hom}^{[n]}(\G,R)$ of sets of
functions from $\G$ to a ring $R$ as follows. For $L:\G\to R$ and $\g \in \G$,
set $L|\g:=L(\g)$ and extend this linearly to all $R[\G]$.  For integers $n\geq 1$, define
$$
\operatorname{Hom}^{[n]}(\G,R) :=\Bigl\{L:\G \to R \, \Big| \, L\big|(\g_1 -I)(\g_2-I) \cdots (\g_n-I) = 0  \quad \text{for all} \quad \g_1,\g_2, \dots,\g_n \in \G \Bigr\}.
$$
We see that $\operatorname{Hom}^{[1]}(\G,R)$ is the space of constant functions. Elements $L$ of $\operatorname{Hom}^{[2]}(\G,R)$ satisfy
\begin{equation*}
    L(\g_1\g_2)-L(\g_1)-L(\g_2)+L(I) = 0 \quad \text{for all} \quad \g_1,\g_2 \in \G,
\end{equation*}
making $\g \mapsto L(\g)-L(I)$ a homomorphism into the additive part of $R$.
Similarly, elements $L$ of $\operatorname{Hom}^{[3]}(\G,R)$ satisfy
\begin{equation*}
    L(\g_1\g_2\g_3)-L(\g_1\g_2)-L(\g_1\g_3)-L(\g_2\g_3)+L(\g_1)+L(\g_2)+L(\g_3)-L(I) = 0 \quad \text{for all} \quad \g_1,\g_2 ,\g_3 \in \G.
\end{equation*}
We may call $L \in \operatorname{Hom}^{[n]}(\G,R)$ an {\em $n$th-order map from $\G$ to $R$}. For example, if $\psi \in \mathcal A^{n}(\G)$ then, for fixed $z_0 \in \H$, $L(\g)$ defined as $\psi(\g z_0)$ is in $\operatorname{Hom}^{[n]}(\G,\C)$.

With this notation,  Proposition \ref{main3} indicates that  $S_\ca \in \operatorname{Hom}^{[2]}(\G,\R/\Z)$ and that  $S^*_\ca$ and $\theta_\ca$ are third-order maps, in $\operatorname{Hom}^{[3]}(\G,\R/\Z)$ and $\operatorname{Hom}^{[3]}(\G,\R)$ respectively. (We will see that $S_\ca(I)=S^*_\ca(I)=\theta_\ca(I)=0$.)
In this way, it is clear that the transformation properties of the modular Dedekind symbols reflect the transformation properties of the Eisenstein series they are derived from, though we don't obtain maps of order greater than $3$. It might be the case that such higher-order maps may be derived from the next term $C_\ca^{(m)}(z)$ in \eqref{emnpl}.


\subsection{Outline}

The paper is organized as follows.  In section 2 we establish further notation and prove
important preliminary material which is required for our investigation.  In section 3 we
present known results for Dedekind sums associated to the full modular group $\text{\rm SL}(2,\Z)$,
originally due to Dedekind, and the congruence subgroups $\G_{0}(N)$ which
are proved in the unpublished thesis \cite{Va}.  In addition, we compute the modular Dedekind
symbols associated to the arithmetic groups $\Gamma_{0}(N)^{+}$ for square-free $N$.  We study the
Laurent expansion of higher-order Eisenstein series in section 4, ultimately defining the
analogue of the Dedekind eta function and modular Dedekind symbols.  In section 5 we prove various properties
of these symbols, completing the proofs of Theorems \ref{main1}, \ref{main2} and Proposition \ref{main3}.
We study how both types of modular Dedekind symbols change under certain involutions in section 6, and use these symmetries to prove the rationality results of Theorem \ref{thm: rationality}.  We conclude with explicit computations for the groups $\G_{0}(11)$ and
$\G_{0}(37)^{+}$.

\vskip .15in
\textbf{Acknowledgements.}
The authors thank Holger Then for his assistance in preparing the example in section 7.2.
We are very grateful  for his computational results and, in general, for the
generosity by which he shares his mathematical insight.

\section{Preliminaries}

\subsection{Basic definitions} \label{basicd}

Our notation is derived primarily from the texts \cite{Iw2},  \cite{Iw} as well as from the articles \cite{OS00} and \cite{JO'S08}.

As stated, $\G \subseteq \slr$ will denote a Fuchsian group of the first kind acting on $\H$. There is a one-to-one correspondence between subgroups of $\slr$ that contain $-I$ and subgroups of $\PSL(2,\R)$, though we do not assume  $-I \in \G$. The image of $\G$ under the projection $\G \to \G/\pm I$ has the following presentation, due to Fricke and Klein, as described in \cite[p. 33]{Iw2} for example. The generators are $2g$ hyperbolic elements $A_1, \dots ,A_g, B_1, \dots, B_g$ along with $e$ elliptic elements $E_1, \dots, E_e$ and $c$ parabolic elements $P_1, \dots, P_c$. They satisfy the relations
\begin{equation} \label{abab}
  [A_1,B_1] \cdots [A_g,B_g]E_1 \cdots E_e P_1 \cdots P_c =I, \qquad E_j^{m_j}=I, \ j=1, \dots,e
\end{equation}
for $[A_i,B_i]:= A_iB_i A_i^{-1}B_i^{-1}$. The genus of $\GH$ is $g$ and the number of inequivalent cusps is $c$ which we assume is positive. As in \cite[Eq. (2.17)]{Iw2}, the volume $V_\G$ of $\GH$ satisfies
\begin{equation}\label{gb}
  \frac{V_\G}{2\pi} = 2g-2+c+\sum_{j=1}^e \left(1-1/m_j\right).
\end{equation}

Let $S_k(\G)$ be the space of holomorphic weight $k$
cusp forms for $\G$, meaning the vector space of holomorphic
functions $f$ on $\H$ which satisfy the transformation property $f(\g z)=j(\g,z)^k f(z) $ for $\g \in \G$
and decay rapidly in each cusp.  As usual, we equip the vector space $S_{k}(\G)$
with the  Petersson inner product
\begin{equation*}
  \s{f}{g}:=\int_\GH y^{k-2}f(z)\overline{g(z)}\, dx dy
\end{equation*}
giving the norm $||f||^2 = \s{f}{f}$. When $k=2$, the space $S_2(\G)$ is $g$ dimensional.

  For the most part, $\G$ is arbitrary
and not necessarily arithmetic.  However, the following two families of arithmetic
subgroups of $\slr$ will provide us with interesting examples.
For  $N$  a positive integer, let
$\G_{0}(N)$ denote the Hecke congruence group of level $N$. This is the subgroup of $\slz$ with matrices having bottom left entry divisible by $N$.
Various properties of $\G_{0}(N)$ are well-documented in the existing literature.
For $N$  square-free, the Helling or moonshine type group of level $N$ is
\begin{equation*}
  \Gamma_0(N)^+ := \left\{\frac 1{\sqrt{e}}\m{a}{b}{c}{d} \in
    \SL(2,\R):\ a,b,c,d,e\in\Z, \ e\mid N,\ e\mid a,
    \ e\mid d,\ N\mid c \right\}
\end{equation*}
 and the corresponding surface $\Gamma_0(N)^+\backslash \H$
possesses only one cusp which is at $\infty$.  We refer to \cite{JST16} and the references it contains for
additional background information on $ \Gamma_0(N)^+$. In particular, for $N=p$ prime, we have the disjoint union
 \begin{equation}\label{disjoint}
   \Gamma_0(p)^+= \Gamma_0(p) \cup \Gamma_0(p)\tau_p \qquad \text{for} \qquad \tau_p :=\m{0}{ -1/\sqrt{p}}{\sqrt{p}}{0}.
 \end{equation}

\subsection{The phase factor}

For any $M, N \in \SL (2,\R)$, recall that we defined the  phase factor
 \begin{equation} \label{w1b}
\omega(M,N)=\bigl(-\log j(MN, z) + \log j(M, N z)+ \log j(N, z)\bigr)/(2\pi i)
\end{equation}
for $z\in \H$ already in \eqref{w1}. One can easily compute
that $j(MN,z) = j(M,Nz)j(N,z)$, so $\omega$ shows the difference in arguments of the
two sides  and takes only  values in $\{-1,0,1\}$.
For example, it is elementary to show that
\begin{equation}\label{w2}
\omega((\smallmatrix 1
& * \\ 0 & 1 \endsmallmatrix ),M)= \omega(M,(\smallmatrix 1
& * \\ 0 & 1 \endsmallmatrix ))=0
\,\,\,\,\,
\text{\rm and}
\,\,\,\,\,
\omega(-I,-I)=1.
\end{equation}

For any matrix $M$ we denote its second row as $(c_M,d_M)$. The right side of \eqref{w1b} is independent of $z\in \H$ and this allowed Petersson in \cite[pp. 44-45]{Pet} to prove the following formula
by making convenient choices for $z$ in different cases.
\begin{prop} \cite{Pet} \label{pet}
For $M$ and  $N \in \SL (2,\R)$ we have
\begin{equation*}
4\omega(M,N) = \begin{cases}
\sgn(c_M)+\sgn(c_N)-\sgn(c_{MN})-\sgn(c_M c_N c_{MN}) & \text{if} \quad c_M c_N c_{MN} \neq 0\\
(\sgn(c_M)-1)(1-\sgn(c_N))  & \text{if} \quad c_M c_N \neq 0, c_{MN}=0 \\
(1-\sgn(d_M))(1+\sgn(c_N)) & \text{if} \quad c_N c_{MN} \neq 0, c_{M}=0 \\
(1+\sgn(c_M))(1-\sgn(d_N)) & \text{if} \quad c_M c_{MN} \neq 0, c_{N}=0 \\
(1-\sgn(d_M))(1-\sgn(d_N)) & \text{if} \quad c_M = c_{MN} = c_{N}=0.
\end{cases}
\end{equation*}
\end{prop}

We give a simplified version of Proposition \ref{pet} in section \ref{slzsym}.
To identify the special case where $c_M=0$ and $d_M<0$ we put
\begin{equation}\label{rho-def}
\rho(M) := \begin{cases} 1 & \text{ if } c_M =0 \text{ and } d_M < 0 \\ 0 & \text{ otherwise. }   \end{cases}
\end{equation}
This notation is useful in the following lemma.

\begin{lemma} \label{lemma: preparation for iota prop}
For $z \in \H$ and any $M \in \SL (2,\R)$ we have the following relations:
\begin{enumerate}
  \item $\log\bigl(\overline{j(M, z)}\bigr)  =  \overline{\log\bigl(j(M, z)\bigr)} + 2\pi i \cdot \rho(M)$;
  \item $\omega (M, M^{-1}) = \rho(M)$.
\end{enumerate}
\end{lemma}

\subsection{The higher-order Eisenstein series}

For $f \in S_{k}(\G)$ we defined $F_\ca(z,f)=F_\ca(z)$ in \eqref{ffff} to be the definite integral $2\pi
i\int_{\ca}^{z}f(w)\,dw$.
Let us write the Fourier expansion of $f$ in the cusp $\cb$ as
\begin{equation}\label{jff}
  j(\sb,z)^{-k}f(\sb z) =  \sum_{n=1}^\infty a_\cb(n) e(nz),
\end{equation}
so then we have the evaluation that
\begin{equation}\label{jff2}
  F_\ca(\sb z) =  2\pi i
\int_\ca^\cb f(w)\, dw+ \sum_{n=1}^\infty \frac{a_\cb(n)}{n}
e(nz).
\end{equation}
For the remainder of this paper we set $f$ to have weight two, meaning $f \in S_{2}(\G)$.
The  modular symbol  associated
to $f$ is the homomorphism from $\G$ to $\C$ given by \eqref{mod}.
 Clearly we  have $\langle\g,f\rangle = F_\ca(\g z)-F_\ca(z)$.
Also $\langle\g,f\rangle$ is necessarily zero on parabolic and elliptic elements $\g$. If $g$ is a holomorphic  weight $2$ modular form for $\G$ and not a cusp form, then $\langle\g,g\rangle$ is well-defined with  \eqref{mod}, though it will no longer be zero on all parabolic  elements.

As stated, the higher-order non-holomorphic Eisenstein series associated to $f$
is formally defined by the series
$$
E^{m,n}_\ca(z,s)= \sum_{\g \in \G_\ca\backslash\G} \langle\g,f\rangle^{m}\overline{\langle\g,f\rangle}^{n} \Im(\sai\g z)^s.
$$
This series was originally introduced by Goldfeld \cite{Go99a,Go99b} in his study of the values taken by $\langle\g,f\rangle$ as $\g$ varies in $\G$.
It converges for $\Re(s) > 1$, admits a meromorphic continuation to all $s \in \C$ and satisfies a functional equation relating values at $s$ and $1-s$; see \cite{OS00} and
\cite{JO'S08}.  The series $E^{m,n}_\ca(z,s)$ is not $\G$-invariant when $m+n>0$ and transforms as described in section \ref{alg-int}.  As shown in \cite{JO'S08}, each series
$E^{m,n}_\ca(z,s)$ may also be represented as a component of a $\G$-invariant section of a unipotent bundle on $\G\backslash \H$.  It is clear that the related function
\begin{equation}\label{qmne_1}
    Q^{m,n}_\ca(z,s) :=  \sum_{\g \in \G_\ca\backslash\G} F_\ca(\g z)^{m} \overline{F_\ca(\g z)}^{n} \Im(\sai \g z)^s
\end{equation}
is $\G$-invariant since replacing $z$ by $\delta z$ for $\delta\in \G$ just reorders the series.

\begin{prop}\label{poles of Em,m-1}

Let $m\geq 1$ be an integer.
\begin{enumerate}
\item The Eisenstein series  $E^{m,m-1}_\ca(z,s)$ possesses a pole at $s=1$ of order $m$. The constant
multiplying the term $(s-1)^{-m}$ in the Laurent series expansion of $E^{m,m-1}_\ca(z,s)$ at $s=1$ is equal to $-mA_{m-1}F_\ca(z)$.
\item The Eisenstein series  $E^{m-1,m}_\ca(z,s)$ possesses a pole at $s=1$ of order $m$. The constant
multiplying the term $(s-1)^{-m}$ in the Laurent series expansion of $E^{m-1,m}_\ca(z,s)$ at $s=1$ is equal to $-mA_{m-1}\overline{F_\ca(z)}$.
\item For  integers $i,$ $j$ with $0 \leq i,j \leq m$ and $(i,j) \not\in \{(m,m), (m,m-1), (m-1,m), (m-1,m-1)\}$,
the Eisenstein series  $E^{i,j}_\ca(z,s)$ possesses a pole at $s=1$ of order at most $m-1$.
\end{enumerate}
\end{prop}

\begin{proof}
We first prove the statement $(i)$; the proof of $(ii)$ then follows by conjugation.
Without loss of generality, we  assume that $\ca =\infty$ and denote $E^{i,j}_{\infty}(z,s)$ simply by $E^{i,j}(z,s)$.
We shall use the results of \cite[Theorems 2.15, 2.16]{PR} where the orders of poles and the leading terms of a closely related series
$$E^{\Re^{a}, \Im^{b}}(z,s) =(-1)^{a+b} \sum_{\g \in \G_{\infty}\backslash\G}\langle\gamma, \Re(f(z)dz) \rangle^{a} \langle\gamma, \Im(f(z)dz) \rangle^{b} \Im(\g z)^s $$
are computed. Note that the series $E^{i,j}(z,s)$ in \cite{PR} are equal to ours with a factor $(-1)^{i+j}$ since their modular symbols
$\s{\g}{f}$ have opposite sign.  Set $n=m-1$ and write
\begin{equation}\label{mnne}
E^{n+1,n}(z,s)= (-1)^{n+1}\sum_{j=0}^{n}\binom{n}{j}\left(E^{\Re^{2j+1}, \Im^{2(n-j)}}(z,s) + i E^{\Re^{2j}, \Im^{2(n-j)+1}}(z,s) \right).
\end{equation}
As shown in \cite[Theorem 2.16]{PR}, for non-negative integers $a,$ $b$ we have the following assertions:
\begin{enumerate}
\item[{\bf(a)}] The series $E^{\Re^{2a}, \Im^{2b}}(z,s)$ has a pole of order $a+b+1$ at $s=1$, and the coefficient of $(s-1)^{-a-b-1}$ in its Laurent expansion is
    \begin{equation} \label{laur}
        \frac{(-4\pi^2)^{a+b}}{V_\G^{a+b+1}} \| f\|^{2a+2b} (2a)!(2b)! \binom{a+b}{a}.
    \end{equation}
\item[{\bf(b)}] The series $E^{\Re^{2a+1}, \Im^{2b+1}}(z,s)$ has a pole of order at most $a+b+1$ at $s=1$.
\end{enumerate}
Applying similar reasoning, we may also deduce from \cite[Theorem 2.15]{PR} the remaining cases:
\begin{enumerate}
\item[{\bf(c)}] The series $E^{\Re^{2a+1}, \Im^{2b}}(z,s)$ has a pole of order $a+b+1$ at $s=1$, and the coefficient of $(s-1)^{-a-b-1}$  is
    $(2a+1) 2\pi i \int_{\infty}^{z} \Re(f(w)dw)$ times the expression \eqref{laur}.
\item[{\bf(d)}] The series $E^{\Re^{2a}, \Im^{2b+1}}(z,s)$ has a pole of order $a+b+1$ at $s=1$, and the coefficient of $(s-1)^{-a-b-1}$  is
    $(2b+1) 2\pi i \int_{\infty}^{z} \Im(f(w)dw)$ times the expression \eqref{laur}.
\end{enumerate}
Using {\bf(c)} and {\bf(d)} in \eqref{mnne}, we find that $E^{n+1,n}(z,s)$ has a pole of order $n+1$ and the coefficient of $(s-1)^{-n-1}$ is
\begin{align*}
    (-1)^{n+1} \sum_{j=0}^{n}\binom{n}{j} & \frac{(-4\pi^2)^{n}}{V_\G^{n+1}} \| f\|^{2n} (2j)!(2(n-j))! \binom{n}{j}\\
    & \quad \times\left[(2j+1)2\pi i \int_{\infty}^{z} \Re(f(w)dw) + i (2(n-j)+1)2\pi i \int_{\infty}^{z} \Im(f(w)dw)   \right]\\
    &= (-1)^{n+1}  \frac{(-4\pi^2)^{n}}{V_\G^{n+1}} \| f\|^{2n} \left(2\pi i \int_{\infty}^{z} f(w)dw \right) \sum_{j=0}^{n} (2j+1) \binom{n}{j} \binom{n}{j} (2j)!(2(n-j))! \\
     &=  - \frac{(4\pi^2)^{n}}{V_\G^{n+1}} \| f\|^{2n} \left(2\pi i \int_{\infty}^{z} f(w)dw \right) (n!)^2 (n+1) 4^n\\
     &=  -(n+1)A_n F_\infty(z),
\end{align*}
as required, where the combinatorial sum is evaluated in Lemma \ref{combl} below.

Now we prove part $(iii)$, recalling the conditions on $i$ and $j$. Since $E^{i,j}(z,s)$ is a linear combination of series $E^{\Re^{a}, \Im^{b}}(z,s)$
for integers $a,$ $b \geq 0$ satisfying $a+b=i+j$, it follows from {\bf(a)} through {\bf(d)}
above that $E^{i,j}(z,s)$ has a pole of order at most $m-1$ except possibly in the cases $(i,j) \in \{(m,m-2),(m-2,m)\}$. Let $n=m-2$ so then we have that
\begin{equation}\label{mnne2}
E^{n+2,n}(z,s)= (-1)^{n}\sum_{j=0}^{n}\binom{n}{j}\left(E^{\Re^{2j+2}, \Im^{2(n-j)}}(z,s) + 2i E^{\Re^{2j+1}, \Im^{2(n-j)+1}}(z,s) - E^{\Re^{2j}, \Im^{2(n-j)+2}}(z,s) \right).
\end{equation}
With {\bf(b)}, the term $E^{\Re^{2j+1}, \Im^{2(n-j)+1}}(z,s)$ in \eqref{mnne2} contributes a pole of order at most $n+1$. Therefore, using {\bf(a)},
$E^{n+2,n}(z,s)$ has a pole of order at most $n+2$ and the coefficient of $(s-1)^{-n-2}$ is
\begin{equation*}
    (-1)^n \frac{(-4\pi^2)^{n+1}}{V_\G^{n+2}} \| f\|^{2n+2} \sum_{j=0}^{n}\binom{n}{j}
    \left[ (2j+2)!(2n-2j)! \binom{n+1}{j+1} - (2j)!(2n-2j+2)! \binom{n+1}{j} \right].
\end{equation*}
Since the sum is
\begin{equation*}
     \sum_{j=0}^{n}\binom{n}{j}\cdot 2(n+1)(2j)!(2n-2j)!\binom{n}{j}
    \Bigl[ (2j+1) - (2(n-j)+1) \Bigr] = 0
\end{equation*}
we see that $E^{n+2,n}(z,s)$ has a pole at $s=1$ of order at most $n+1$. It follows that $E^{m,m-2}(z,s)$ and $E^{m-2,m}(z,s)$ have poles of order at most $m-1$ as we wanted.
The proof is complete.
\end{proof}

In the above computations, we needed the following combinatorial lemma.

\begin{lemma} \label{combl} For any integer $n \geq 0$, we have
\begin{equation}\label{uy}
    \sum_{j=0}^n (2j+1)\binom{2j}{j} \binom{2(n-j)}{n-j} = (n+1)4^n.
\end{equation}
\end{lemma}
\begin{proof}
Starting with the generating function
\begin{equation*}
    g(x):=(1-4x^2)^{-1/2} = \sum_{j=0}^\infty \binom{-1/2}{j}(-4x^2)^j = \sum_{j=0}^\infty \binom{2j}{j}x^{2j},
\end{equation*}
we note first that the coefficient of $x^{2n}$ in the power series for $g(x) \cdot g(x)$ is  $4^n$ and so we obtain
\begin{equation}\label{coid}
    \sum_{j=0}^n \binom{2j}{j} \binom{2(n-j)}{n-j} = 4^n.
\end{equation}
This identity \eqref{coid} was used in the proof of \cite[Thm. 2.18]{PR}.
Continuing, for the proof of \eqref{uy} we write
\begin{equation*}
    g^3(x) = \frac{d}{dx} \bigl(x\cdot g(x)\bigr) =  \sum_{j=0}^\infty (2j+1)\binom{2j}{j}x^{2j}.
\end{equation*}
The coefficient of $x^{2n}$ in $g^3(x) \cdot g(x)$ is easily seen to be $(n+1)4^n$, which completes
the proof of the lemma.
\end{proof}

For any two cusps $\ca, \cb$ of $\Gamma$,  the series $E^{m,m}_\ca( \sb z,s)$ for $m\geq 1$ admits the Fourier expansion
\begin{equation} \label{Four. exp E^11}
E^{m,m}_\ca( \sb z,s)=\phi^{m,m}_{\ca \cb}(0, s) y^{1-s}+\sum_{k\neq 0}\phi^{m,m}_{\ca \cb}(k,s)W_s(kz),
\end{equation}
see \cite[Eq. (2.4)]{JO'S08}.
As is commonplace, $W_s(z)$ denotes the classical Whittaker function which is defined for $z=x+iy \in \H$ by
$$
W_s(z)= 2\sqrt{y}K_{s-1/2}(2\pi y) e(x),
$$
where $K_{s}$ is the $K$-Bessel function.  The Whittaker function is
extended to the lower half-plane by imposing the symmetry $W_s(z)=W_s(\overline{z})$.
The coefficients $\phi^{m,m}_{\ca \cb}(k,s)$, for $k\in\Z$, admit meromorphic continuations to the whole complex plane with  poles at $s=1$ of order at most $m+1$. Let
\begin{equation} \label{Laur. exp phi^11}
\frac{\phi^{m,m}_{\ca \cb}(k,s)}{A_m} = \frac{a_{\ca \cb}^{(m)}(k)}{(s-1)^{m+1}}+\frac{b_{\ca \cb}^{(m)}(k)}{(s-1)^m}+O\left(\frac{1}{(s-1)^{m-1}}\right)
\end{equation}
be the Laurent series expansion of $\phi^{m,m}_{\ca \cb}(k,s)$ at $s=1$.

\begin{prop} \label{prop: properties of coeff b ab}
Let $m\geq 0$ be an integer.
The coefficients $b_{\ca \cb}^{(m)}(k)$, with $k\neq 0$, satisfy the bound
\begin{equation} \label{b_ab(m) bound}
b_{\ca \cb}^{(m)}(k)\ll \log^{2m}|k| \cdot |k|^{1+\epsilon},
\end{equation}
for arbitrarily small $\epsilon>0$, where the implied constant is independent of $k$.
\end{prop}
\begin{proof}
Theorem 2.3 of \cite{JO'S08} with $m=n$ yields that for a compact set $S\subset \C$ there exist a holomorphic function\footnote{It should have been stated in this theorem that $\xi_S^{m,m}(s) \not\equiv 0$ which is clear from its construction.}  $\xi_S^{m,m}(s)$ such that for all $s\in S$ and all $k\neq 0$ we have
\begin{equation}\label{phi^11 bound}
    \phi^{m,m}_{\ca \cb}(k,s) \xi_S^{m,m}(s) \ll (\log^{2m}|k| +1)(|k|^{\sigma} + |k|^{1-\sigma}),
\end{equation}
 where $\sigma = \Re (s)$ and the implied constant depends solely on $f$ and $\Gamma$. Let $S$ be the closed disc around $s=1$ with radius $\epsilon$ and let $0<\delta<\epsilon$ be such that $\xi_S^{m,m}(s)$ is non-vanishing on the circle $|s-1|=\delta$. Since
 $$
b_{\ca \cb}^{(m)}(k) = \frac{1}{2\pi i A_m}\int\limits_{|s-1|=\delta} \phi^{m,m}_{\ca \cb}(k,s)(s-1)^{m-1} ds,
 $$
 applying the bound \eqref{phi^11 bound} we immediately deduce \eqref{b_ab(m) bound}.
\end{proof}

Proposition \ref{prop: properties of coeff b ab} above ensures that the series we are considering, such as $H^{(m)}_\ca(z)$ in \eqref{hamz}, are absolutely convergent for $z \in \H$.


\section{The Kronecker limit formula and modular Dedekind symbols}
In this section we recall results regarding the Kronecker limit formula and Dedekind sums for first-order Eisenstein series
associated to an arbitrary non-compact Fuchsian group of the first kind, denoted by $\Gamma$.  For this general setting, the material we present
is mostly due to Goldstein, see \cite{Gn1}.  We then consider three special cases: the full modular group $\slz$,
the congruence groups $\Gamma_0(N)$ and the moonshine type groups $\Gamma_0(N)^+$. By now, the first setting is classical, and
the remaining two cases employ results from the  thesis \cite{Va} and the articles \cite{JST16,JST3}, respectively.
For completeness, we note that related computations for the groups $\Gamma(N)$ are given in \cite{Tak}.

\subsection{Modular Dedekind symbols associated to a general Fuchsian group $\Gamma$}
The non-holomorphic, parabolic Eisenstein series associated to a cusp $\ca$ is defined
as
$$
E_\ca(z,s)=\sum_{\g \in \G_\ca\backslash\G} \Im(\sai\g z)^s,
$$
converging for $\Re (s)>1$. It has a meromorphic continuation
to all $s$ in $\C$ and a simple pole at $s=1$ with
residue $V_{\G}^{-1}$. For all these properties see for example \cite[Chapters 3, 6]{Iw}.
In the notation of \eqref{E m,n}, \eqref{emnpl} and \eqref{amdef}, we have $E_\ca(z,s) = E^{0,0}_\ca(z,s)$ and write
\begin{equation}\label{dhk}
\frac 1{A_0} E^{0,0}_\ca(z,s)=\frac{1}{s-1}+B^{(0)}_\ca(z)+
O\left(s-1\right) \quad \text{as} \quad s \to 1.
\end{equation}
 The function $B^{(0)}_\ca(\sb z)$ is invariant under $z \mapsto z+1$ and, as in \cite{Gn1}, we obtain the Fourier expansion
\begin{align}
    B^{(0)}_\ca(\sb z) & = - \log y +\delta_{\ca \cb} V_\G \cdot y + b^{(0)}_{\ca \cb}(0) + \sum_{k<0}b^{(0)}_{\ca \cb}(k)e(k\overline{z}) +
\sum_{k>0}b^{(0)}_{\ca \cb}(k)e(k z) \label{bexpn0}\\
& = - \log y  + b^{(0)}_{\ca \cb}(0) + 2\Re(H^{(0)}_{\ca \cb}(z)) \label{abh}
\end{align}
where $H^{(0)}_{\ca \cb}(z)$ is the holomorphic function defined by the series
\begin{equation}\label{h0expn}
H^{(0)}_{\ca \cb}(z)=-\delta_{\ca \cb}  V_\G \cdot i z/2+ \sum_{k=1}^\infty b^{(0)}_{\ca \cb}(k)e(kz).
\end{equation}
Note that in the above assertion we have used that $\overline{E_\ca(z,s)}-E_\ca(z,\overline{s})=0$, which
implies that $\overline{B^{(0)}_\ca(z)}=B^{(0)}_\ca(z)$ and, therefore,
\begin{equation*}\label{kn}
b^{(0)}_{\ca \cb}(-n)=\overline{b^{(0)}_{\ca \cb}(n)}
\end{equation*}
for all $n \in \Z$.  Let $\g_\cb := \sb^{-1} \g \sb$ for $\g \in \G$ and define the holomorphic function of $z$
$$
  K_{\ca \cb}(\g,z):=  \log j(\g_\cb,z) + H^{(0)}_{\ca \cb}(\g_\cb z)- H^{(0)}_{\ca \cb}(z).
$$
Since $B^{(0)}_\ca(z)$ is $\G$ invariant, it follows from \eqref{abh} that $\Re( K_{\ca \cb}(\g,z))=0$, and hence
$$
 K_{\ca \cb}(\g,z) = -2\pi i S_{\ca \cb}(\g)
$$
for some $S_{\ca \cb}(\g) \in \R$ which is independent of $z$.
We have shown that
\begin{equation} \label{h0}
    H^{(0)}_{\ca \cb}(\g_\cb z)- H^{(0)}_{\ca \cb}(z) = - \log j(\g_\cb,z) -2\pi i S_{\ca \cb}(\g)
    \quad \text{for all} \quad \g \in \G.
\end{equation}

\begin{defs}{\rm
The {\em modular Dedekind symbol associated to $\G$ and the cusps $\ca$ and $\cb$} is the real-valued function $S_{\G, \ca \cb}(\g) = S_{\ca \cb}(\g)$ satisfying \eqref{h0}. The {\em  Dedekind eta function associated to $\G$ and the cusps $\ca$ and $\cb$} is
\begin{equation*}
    \eta_{\G, \ca \cb}(z) = \eta_{\ca \cb}(z):= \exp\left(-\frac 12 H^{(0)}_{\ca \cb}(z)\right).
\end{equation*}}
\end{defs}

We see from \eqref{h0} that $\eta_{\ca \cb}(z)$ is a holomorphic function that transforms with weight $1/2$ for $\sb^{-1} \G \sb$ up to a factor of modulus $1$. Writing $\log \eta_{\ca\cb}(z)$ for $-H^{(0)}_{\ca \cb}(z)/2$ we have
\begin{equation} \label{etatran}
\log\eta_{\ca\cb}(\g_\cb z)=  \log\eta_{\ca\cb}(z) + \frac{1}{2} \log j(\g_\cb, z) + \pi i S_{\ca\cb}(\g)\quad \text{for all} \quad \g \in \G.
\end{equation}
In particular, let us take $\cb = \ci$ and write $H^{(0)}_{\ca}$ for $H^{(0)}_{\ca \ci}$, $S_{\ca}$ for $S_{\ca \ci}$ and $\eta_{\ca}$ for $\eta_{\ca \ci}$ as in the introduction.
Then
\begin{equation} \label{etatran2}
\log\eta_{\ca}(\g z)=  \log\eta_{\ca}(z) + \frac{1}{2} \log j(\g, z) + \pi i S_{\ca}(\g)
\quad \text{for all} \quad \g \in \G
\end{equation}
and $\eta_{\ca}(z)$  transforms with weight $1/2$ with respect to $\G$ and the multiplier system $e^{\pi i S_{\ca}(\g)}$:
\begin{equation*} \label{ntr}
    \eta_{\ca}(\g z) = j(\g,z)^{1/2} e^{\pi i S_{\ca}(\g)} \eta_{\ca}(z) \quad \text{for all} \quad \g \in \G.
\end{equation*}
From the identity
$$
H^{(0)}_{\ca}(\g \tau z) - H^{(0)}_{\ca}(z)= \left(H^{(0)}_{\ca}(\g (\tau z)) - H^{(0)}_{\ca}( \tau z)\right) + \left(H^{(0)}_{\ca}(\tau z) - H^{(0)}_{\ca}(z)\right),
$$
\eqref{h0} also implies
\begin{equation}\label{ee}
S_\ca(\g \tau)=S_\ca(\g)+S_\ca(\tau) + \omega(\g , \tau),
\end{equation}
which proves part (i) of Theorem 1.

It follows from \eqref{ee} that $S_\ca(I)=0$ and also, since $\omega(-I,-I)=1$, that
\begin{equation}\label{ee2}
S_\ca(-I)=-1/2 \quad \text{if} \quad -I \in \G.
\end{equation}
The identities \eqref{ee} and \eqref{ee2}  show that $e^{\pi i S_{\ca}(\g)}$ is indeed a multiplier system of weight $1/2$ for $\G$; see for example \cite[Sect. 2.6]{Iw2} for more on multiplier systems.

A relation that we will need later is another consequence of \eqref{ee}:
 for all $\g_1,\g_2,\g_3 \in \G$ we have
\begin{multline} \label{triple prod. for xi}
S_\ca(\g_1\g_2\g_3)-S_\ca(\g_1\g_2) -S_\ca(\g_1\g_3)-S_\ca(\g_2\g_3)+ S_\ca(\g_1)+
S_\ca(\g_2)+S_\ca(\g_3)\\ =\omega(\g_1\g_2,\g_3) -\omega(\g_1,\g_3)- \omega(\g_2,\g_3).
\end{multline}

The  results in this subsection may be compared with those from Asai in \cite{Asai} for $\slz$ and Goldstein  for general $\G$ in \cite{Gn1}. Errors in  \cite{Gn1} are corrected in \cite{Gn2} and \cite{Tak}.  Goldstein labels the cusps as $\kappa_i$ and the functions he studies, $\eta_{\G,i}$ and $S_{\G,i}$, correspond to our $\eta_{\G,\ca\ca}$ and  $S_{\G,\ca \ca}$ respectively. In our approach we work with the general case of $\eta_{\G,\ca\cb}$ and  $S_{\G, \ca \cb}$ for any pair of  cusps $\ca$ and $\cb$.

\subsection{Modular Dedekind symbols and Dedekind sums associated to  $\slz$ } \label{slzsym}
We illustrate the theory of the previous section with $\G = \slz$, the modular group. It has a single cusp which we may take to be at infinity. 
We have $1/A_0=V_\G=\pi/3$ and so \eqref{dhk} becomes
\begin{equation*}
    \frac \pi{3} E^{0,0}_\ci(z,s)=\frac{1}{s-1}+B^{(0)}_\ci(z)+
O\left(s-1\right) \quad \text{as} \quad s \to 1.
\end{equation*}
Then, following from the limit formula \eqref{KLT},
\begin{align*}
    B^{(0)}_\ci(z) & = -\log y +2 - 24\zeta'(-1)-2\log (4\pi)+ 2 \Re(H^{(0)}_\ci(z)),\\
    H^{(0)}_\ci(z) & = -\frac{\pi i z}{6}+\sum_{k=1}^\infty 2\sigma_{-1}(k) e(k z)
\end{align*}
where $\sigma_{-1}(k):=\sum_{d|k}d^{-1}$ and $\eta(z)=\eta_\ci(z)=\exp(-H^{(0)}_\ci(z)/2)$ is the usual Dedekind  eta function \eqref{etad}. Therefore, what we mean by $\log \eta(z)$ is really
\begin{equation} \label{logetadef}
\log \eta(z) := -H^{(0)}_\ci(z)/2 = \frac{\pi i z}{12}-\sum_{k=1}^\infty \sigma_{-1}(k) e(k z).
\end{equation}
It satisfies the transformation formula
\begin{equation} \label{eta transf}
\log \eta(\g z) = \log \eta(z) +\frac{1}{2}\log j(\g,z) + \pi i S(\g)
\end{equation}
as we already saw in \eqref{log_eta} with $S(\g)=S_\ci(\g)$ given by \eqref{Dsum}, \eqref{Dsum2} for $\g=
(\smallmatrix a & b \\ c & d \endsmallmatrix )$.

Dedekind's original formulation for \eqref{eta transf} was
\begin{equation} \label{eta transf2}
\log \eta(\g z) = \log \eta(z) +\frac{1}{2}\log \left(\frac{j(\g,z)}i\right) + \pi i \left( \frac{a+d}{12c}-s(d,c)\right)
\end{equation}
for $c>0$. See for example \cite[Eq. (57b)]{RG72} and also \cite[p. 52]{Apo} where \eqref{eta transf2} is proved. It is straightforward to obtain the case when $c<0$ from \eqref{eta transf2} by changing the sign of all the matrix elements in $\g$ and using that $(-\g)z=\g z$. Altogether this yields \eqref{Dsum}. When $c=0$ we must have $d=\pm 1$ and $\g z=z+b/d$. Then \eqref{Dsum2} follows from
\begin{equation*}
    \log \eta(z+b/d) = \log \eta(z) +\frac{\pi i b}{12 d}
\end{equation*}
which is \cite[Eq. (57a)]{RG72}.

For the generators of $\SL(2,\Z)$, a computation with \eqref{Dsum}, \eqref{Dsum2} shows that $S((\smallmatrix 1 & 1 \\ 0 & 1 \endsmallmatrix ))=1/12$ and $S((\smallmatrix 0 & -1 \\ 1 & 0 \endsmallmatrix ))=-1/4$. It now follows from \eqref{ee} that $12 S(\g)$ is always an integer and that $e^{\pi i S(\g)}$ is a $24$th root of unity for all $\g \in \SL(2,\Z)$, as Dedekind discovered.

Rademacher gave a slightly different formulation of \eqref{eta transf}. In \cite[Eq. (60)]{RG72} (see also \cite[Chap. IX]{La95}) he wrote
\begin{equation} \label{eta transf3}
\log \eta(\g z) = \log \eta(z) +\frac{1}{2}\sgn ( c )^2 \log \left(\frac{cz+d}{i \sgn ( c )}\right) + \frac{\pi i}{12}\Phi(\g)
\end{equation}
where
$$
\sgn ( c ) := \begin{cases} c/|c| & \text{ if } c\neq 0 \\ 0 & \text{ if } c=0\end{cases}
$$
and we understand the middle term on the right of \eqref{eta transf3} to be $0$ if $c=0$. Perhaps a clearer way to state \eqref{eta transf3}, allowing an easier comparison with \eqref{eta transf}, is as
$$
\log \eta(\g z) = \log \eta (z) +\frac 12 \log j(\g, z) -\frac {\pi i}4 R(\g) +\frac{\pi i}{12} \Phi(\g)
$$
for
$$
R((\smallmatrix a & b \\ c & d \endsmallmatrix)) :=
\begin{cases} \sgn(c) & \text{ if } c\neq 0 \\ \sgn(d) - 1 & \text{ if } c=0.
\end{cases}
$$
With this arrangement, $\Phi(\g)$ has the advantage of remaining unchanged when $\g$ is replaced by $-\g$.
In \cite[pp. 49-53]{RG72} it is proved that $\Phi(\g) \in \Z$ and
\begin{equation}\label{phi3}
\Phi(\g \tau)=\Phi(\g )+\Phi( \tau)-3 \sgn (c_\g c_\tau c_{\g\tau})
\end{equation}
for all $\g,$ $\tau \in \SL(2,\Z)$ where $c_M$ indicates the lower left entry of $M$.
Rademacher also showed that the variant $\Psi(\g):=\Phi(\g )-3\sgn(c(a+d))$ has nice group theoretic properties. Ghys discovered in \cite{Ghy} that
$\Psi(\g)$ can be interpreted as the linking number of a modular knot, associated to primitive hyperbolic $\g$, and a certain trefoil knot related to $\slz\backslash \slr$.

Using $S(\g)=-R(\g)/4+\Phi(\g)/12$ to  combine (\ref{ee}) and (\ref{phi3}) gives the formula \eqref{omg} below for all $\g, \tau$ in $\slz$. In fact we can show that  \eqref{omg} is valid in $\slr$, giving an appealing reformulation of Proposition \ref{pet}.

\begin{prop} \label{cases}
For all $\g, \tau$ in $\slr$ set $\mathbf{c} =(\sgn ( c_\g), \sgn ( c_\tau), \sgn ( c_{\g \tau}))$. Then
\begin{equation}\label{omg}
\omega(\g,\tau)=\begin{cases}
1 & \textrm{ if \ } \mathbf{c} = (1,1,-1)  \text{ or } (0,1,-1) \text{ or } (1,0,-1)\\
1 & \textrm{ if \ } \mathbf{c} = (0,0,0) \textrm{ and } \sgn( d_\g)=\sgn( d_\tau )=-1 \\
-1 & \textrm{ if \ } \mathbf{c} = (-1,-1,1) \text{ or } (-1,-1,0) \\
0 & \textrm{ otherwise. }
\end{cases}
\end{equation}
\end{prop}
\begin{proof}
There are $27$ possibilities for $\mathbf{c}$. The $6$ cases with exactly two zero entries cannot occur. We have by Proposition \ref{pet} that the first three lines of \eqref{omg} are true, showing when $\omega(\g,\tau) = \pm 1$. In all the remaining cases we may use Proposition \ref{pet} to show that $\omega(\g,\tau) = 0$. For example, when $\mathbf{c} = (1,0,1)$ we have $4\omega(\g,\tau) = (1+\sgn(c_M))(1-\sgn(d_N))$. Also $c_M a_N = c_{MN}$ with $a_N$ the top left entry of $N$. Since $a_N$ and $d_N$ have the same sign, it follows that $d_N>0$ and so $\omega(\g,\tau) =0$.
\end{proof}

Asai gives a similar result in \cite[Thm. 2]{Asai} though his logarithms take arguments in the range $[-\pi,\pi)$.

We finally note that the famous reciprocity law for $s(c,d)$ follows from \eqref{ee} and the fact that $(\smallmatrix 0 & -1 \\ 1 & 0 \endsmallmatrix ) \in \slz$. To see this,  start with
\begin{equation} \label{sabc}
    S\left((\smallmatrix a & b \\ c & d \endsmallmatrix )(\smallmatrix 0 & -1 \\ 1 & 0 \endsmallmatrix )\right) = S\left((\smallmatrix a & b \\ c & d \endsmallmatrix )\right)+S\left((\smallmatrix 0 & -1 \\ 1 & 0 \endsmallmatrix )\right)+\omega\left((\smallmatrix a & b \\ c & d \endsmallmatrix ),(\smallmatrix 0 & -1 \\ 1 & 0 \endsmallmatrix )\right)
\end{equation}
for relatively prime positive integers $c$ and $d$. As we saw above, $S\left((\smallmatrix 0 & -1 \\ 1 & 0 \endsmallmatrix )\right)=-1/4$.  The last term in \eqref{sabc} is $0$ as  can be seen from Proposition \ref{cases}. Therefore
\begin{equation*}
    S\left((\smallmatrix b & -a \\ d & -c \endsmallmatrix )\right) = S\left((\smallmatrix a & b \\ c & d \endsmallmatrix )\right) -1/4
\end{equation*}
and applying \eqref{Dsum} then proves $s(c,d)+s(d,c)=(c/d+1/(cd)+d/c)/12-1/4$. See \cite[Lemma 8]{Asai} for an equivalent argument.

\subsection{Modular Dedekind symbols associated to $\Gamma_0(N)$} \label{vass}

The eta function and modular Dedekind symbol associated to the Hecke congruence groups $\Gamma_0(N)$ were described in the unpublished PhD thesis \cite{Va}. For simplicity we write $\eta_{N}$ for $\eta_{\Gamma_0(N),\ci}$ and $ S_N(\g)$ for $S_{\Gamma_0(N),\ci}$. Define the products $\alpha_N:=\prod_{p\mid N} 1/(1-p^{-1})$ and  $\beta_N:=\prod_{p\mid N}(1-p^{-2})/(1-p^{-1})$ where $p$ is prime, and let $\mu$ denote the M\"obius function. Then it is proved in Corollary 3.2.1 and Theorem 3.3.1 of \cite{Va} that
$$
\eta_{N}(z)= \exp\Bigl({\alpha_N \sum_{v \mid N}  \frac{ v}{N} \cdot \mu(N/v)} \cdot \log \eta(vz) \Bigr)
$$
for $\log \eta(z) $ given by \eqref{logetadef}.

Furthermore, in \cite[Thm. 4.1.3, Prop. 4.1.1]{Va} it is shown that,  for
$\g= \left(\smallmatrix a & b \\ c & d \endsmallmatrix \right) $,
\begin{equation}\label{xi N for comgruence}
 S_N(\g)=
 \frac{N(a+d)}{12c} \beta_N  - \frac{c}{4|c|} - \frac{c}{|c|}\alpha_N \sum_{v\mid N} \frac{\mu(v)}{v}s\left( d, \frac{v|c|}{N}\right) \quad  \text{  if } \quad c\neq 0
\end{equation}
and
\begin{equation}\label{xi N for comgruence2}
  S_N(\g)= \frac{Nb}{12d}\beta_N + \frac{d/|d|-1}{4} \quad  \text{  if } \quad  c= 0,
\end{equation}
where $s(h,k)$ is the usual Dedekind sum  \eqref{dede}.

\subsection{Modular Dedekind symbols associated to $\Gamma_0(N)^+$} \label{helling}

The Dedekind eta function associated to $\G_{0}(N)^{+}$ for $N$ a square-free, positive integer with $r$ prime factors was evaluated in \cite{JST16}. Writing $\eta_{N}^+(z)$ for $\eta_{\G_{0}(N)^{+},\ci}(z)$, it follows from \cite[Thm. 12]{JST16} that
\begin{equation*}
   \eta_{N}^+(z) = \exp\Bigl(2^{-r} \sum_{v \mid N} \log \eta(vz)\Bigr)
\end{equation*}
with $\log \eta(z) $ given by \eqref{logetadef}.

 Writing $S_{N}^+(\g)$ for $S_{\G_{0}(N)^{+},\ci}(\g)$,
it is also shown in \cite[Thm. 16]{JST16}  that $\exp(\pi i S_{N}^+(\g))$ is an $\ell_N$th root of unity where
$$
\ell_N = 2^{1-r}\textrm{lcm}\Big(4,\ 2^{r-1}\frac{24}{(24,\sigma(N))}\Big).
$$
In this formula, $\textrm{lcm}$ denotes the least common multiple  and $\sigma(N)$ is the sum of the divisors of $N$.

Let us now compute $S_N^+(\g)$, for $\g\in\Gamma_0(N)^+$ in terms of the classical modular Dedekind symbols $S$ on $\slz$.
For any $\gamma \in\Gamma_0(N)^+$, either $\gamma \in \Gamma_0(N)$ or there exists a divisor $v>1$ of $N$ such that
$$
\g= \left(
      \begin{array}{cc}
        a\sqrt{v} & b/\sqrt{v} \\
        cN/\sqrt{v} & d\sqrt{v} \\
      \end{array}
    \right)
$$
for some integers $a,b,c,d$ such that $adv-bcN/v=1$. In this case, $\g^2 \in \Gamma_0(N)$.

Since $S_N^+(\gamma) = \frac{1}{2}(S_N^+(\gamma ^2) - \omega(\g,\g))$, we see that in order to compute $S_N^+(\g)$ for any $\gamma \in\Gamma_0(N)^+$,
it is sufficient to compute $S_N^+(\g)$ for $\gamma \in\Gamma_0(N)$.
To do so, we begin with the following proposition.

\begin{prop} \label{prop plus: S N}
Let $N$ be a square-free, positive integer with $r$ prime factors. For any $\gamma = \left(\smallmatrix a & b \\ cN & d \endsmallmatrix \right)  \in\Gamma_0(N)$ we have
\begin{equation}\label{snpx}
S_N^+(\g) = \frac{1}{2^r}\sum_{v \mid N} S(\gamma_v) \qquad \text{where} \qquad \gamma_v = \left(\smallmatrix a & bv \\ cN/v & d \endsmallmatrix \right).
\end{equation}
\end{prop}
\begin{proof}
From the definition of $\gamma_v$ it is immediate that $\g_v\in \slz$ and $\g_v(vz)=v \cdot \g (z)$.
Furthermore, for any divisor $v$ of $N$, we have that $j(\g,z)=j(\g_v, vz)$. The number of divisors of $N$ is $2^r$ and
therefore
$$
\log j(\gamma, z)= \frac{1}{2^r}\sum_{v\mid N}\log j(\g_v, vz).
$$
As a result, we can write the transformation rule \eqref{etatran2} for the function $\eta_N^+$ as
$$
\frac{1}{2^r}\sum_{v\mid N} \log \eta(\g_v(vz))= \frac{1}{2^r}\sum_{v\mid N} \log \eta(vz) + \frac{1}{2^r}\sum_{v\mid N}\frac{1}{2} \log j(\g_v, vz) + i\pi S_N^+(\g).
$$
The transformation rule \eqref{eta transf} for the classical Dedekind eta function now implies \eqref{snpx}.
\end{proof}

From Proposition \ref{prop plus: S N} we see that an evaluation of the modular Dedekind symbol for $\G_{0}(N)^{+}$
reduces to the calculation of the phase function $\omega(\g,\g)$ for any $\g \in \Gamma_0(N)^+$:
\begin{equation} \label{S N +}
S_N^+(\g) = \frac{1}{2}\left(S_N^+(\g^2) - \omega(\g,\g)\right)= \frac{1}{2}\Bigl(\frac{1}{2^r}\sum_{v \mid N} S\left((\gamma^2)_v \right) - \omega(\g,\g)\Bigr).
\end{equation}
Let $(c_\g,d_\g)$ be the bottom row of $\g$.
An application of Proposition \ref{cases} allows us to find $\omega(\g,\g)$ explicitly. In the notation of that proposition, $\mathbf{c} =(\sgn ( c_\g), \sgn ( c_\g), \sgn ( c_\g \cdot \mathrm{Tr}(\g)))$ where $\mathrm{Tr}(\g)$ denotes the trace of $\g$.

\begin{lemma} \label{lemma omega}
For  $\gamma \in \slr$ and the above notation we have
$$
\omega(\g,\g)= \begin{cases}
                   1 & \text{   if  } c_{\g}>0 \text{   and  }  \mathrm{Tr}(\g)<0 \\
                    1 & \text{   if  } c_{\g}=0 \text{   and  }  d_\g<0 \\
                   -1 & \text{   if  } c_{\g}< 0 \text{   and  }  \mathrm{Tr}(\g)\geq 0\\
                   0 & \text{otherwise.}
                 \end{cases}
$$
\end{lemma}

Combining equation \eqref{S N +} and Lemma \ref{lemma omega} with the formulas \eqref{Dsum} and \eqref{Dsum2} for the classical symbols $S$ allows $S_N^+$ to be evaluated easily.

When $N=p$ is prime, a slightly simpler method is available because $\Gamma_0(p)^+= \Gamma_0(p) \cup \Gamma_0(p)\tau_p$ as in \eqref{disjoint}.
The next proposition provides an explicit expression for $S_p^+(\g)$ in terms of the entries of an arbitrary matrix $\g \in\Gamma_0(p)^+$.

\begin{prop} \label{prop: S p +}
Let $p$ be a prime and let $\g=\left(
         \begin{array}{cc}
           a\sqrt{v} & b/\sqrt{v} \\
           cp/\sqrt{v} & d\sqrt{v} \\
         \end{array}
       \right)\in\Gamma_0(p)^+$, where $a,b,c,d \in\Z$.
\begin{itemize}
\item[(i)] If $v=1$, then
$$
S_p^+(\g)=
            \begin{cases}
              \frac{a+d}{24cp}(p+1)-\frac{c}{4|c|}-\frac{c}{2|c|}\bigl(s(d,|c|)+s(d,|cp|)\bigr), & \text{  if  }c\neq 0; \\
              \frac{b}{24d} (p+1) + \frac{d/|d|-1}{4}, & \text{  if  }c= 0.
            \end{cases}
$$
\item[(ii)] If $v=p$, then
$$
S_p^+(\g)=
 \begin{cases}
              \frac{b-c}{24dp}(p+1)+\frac{d}{4|d|}+\frac{d}{2|d|}(s(c,|d|)+s(c,|dp|))-\frac{1}{4}, & \text{  if  } d>0 \text{  or  } (d<0 \wedge c\geq 0); \\
              \frac{b-c}{24dp}(p+1)+\frac{d}{4|d|}+\frac{d}{2|d|}(s(c,|d|)+s(c,|dp|))+\frac{3}{4}, & \text{  if  } d<0 \wedge c<0; \\
               \frac{a}{24c}(p+1) - \frac{1}{4}, & \text{  if  }d=0 \wedge c>0; \\
             \frac{a}{24c}(p+1) + \frac{1}{4}, & \text{  if  }d=0 \wedge c<0.
            \end{cases}
$$
Here $s(h,k)$ is the classical Dedekind sum \eqref{dede}.
\end{itemize}
\end{prop}

\begin{proof} Part (i) is an obvious consequence of Proposition \ref{prop plus: S N} and relations  \eqref{Dsum} and \eqref{Dsum2}.

In order to prove part (ii) we start with the decomposition $\gamma = \g(p) \tau_p$ where $\g(p)= \left(\smallmatrix -b & a \\ -dp & c \endsmallmatrix \right).$
 Since $\tau_p^2=-I$, it easily follows from \eqref{mn-s1} that $S_p^+(\tau_p)=-1/4$.
Another application of \eqref{mn-s1} provides
\begin{equation}\label{partial_eval}
S_p^+(\g)= S_p^+(\g(p))-\frac{1}{4} + \omega(\g(p), \tau_p).
\end{equation}
The first term on the right hand side of (\ref{partial_eval}) may be found with part (i).
An application of Proposition \ref{cases} with $\mathbf{c} =(\sgn ( -dp), \sgn ( \sqrt{p}), \sgn ( c\sqrt{p}))$ shows
$$
\omega(\g(p), \tau_p)=\begin{cases}
                          1 & \text{if  } d\leq0 \text{  and  } c<0, \\
                          0 & \text{otherwise}.
                        \end{cases}
$$
Hence part (ii) follows upon combining the expressions for $S_p^+(\g(p))$ and $\omega(\g(p), \tau_p)$.
\end{proof}


\section{The higher-order Kronecker limit formula}

Recall the definitions of $A_m$ and $V_f$ in \eqref{amdef} and \eqref{V f def}.
In \eqref{emnpl}  the Laurent expansion of $E^{m,m}_\ca(z,s)$ at $s=1$ was expressed by
$$
\frac 1{A_m} E^{m,m}_\ca(z,s)=\frac{1}{(s-1)^{m+1}}+\frac{B^{(m)}_\ca(z)}{(s-1)^m}+
\frac{C_\ca^{(m)}(z)}{(s-1)^{m-1}}+O\left(\frac{1}{(s-1)^{m-2}}\right).
$$

We can now prove basic properties of $B_\ca^{(m)}$.

\begin{theorem} \label{bb m} For $m$ a non-negative integer:
\begin{enumerate}
\item $B_\ca^{(m)}(z) \in \R$ for all $z\in \H$,
\item $\Delta B_\ca^{(m)}(z) = -1$,
\item $B_\ca^{(0)}(z)$ is $\G$-invariant and $\displaystyle B_\ca^{(m)}(z)-V_f |F_\ca(z)|^2 $ is $\G$-invariant if $m\geq 1$.
\end{enumerate}
\end{theorem}
\begin{proof} The proof of (i) follows from the identity $ E^{m,m}_\ca(z,s) -\overline{ E^{m,m}_\ca(z, \overline s)}=0$ which
is valid by analytic continuation at $s=1$.
For part (ii) we have
\begin{align*}
  \Delta B_\ca^{(m)}(z)  &=  \Delta  \lim_{s\rightarrow1} \left((s-1)^m \frac{E^{m,m}_\ca(z,s)}{A_m} -\frac{1}{s-1}\right) \\
   &=   \lim_{s\rightarrow1} \Delta \left((s-1)^m \frac{E^{m,m}_\ca(z,s)}{A_m} -\frac{1}{s-1}\right) \\
    &=  \lim_{s\rightarrow 1}\left(-s(s-1)^{m+1} \frac{E^{m,m}_\ca(z,s)}{A_m}\right)
   = -1.
\end{align*}
We finally  prove part (iii). Clearly  $B_\ca^{(0)}(z)$ is $\G$-invariant since $E_\ca^{0,0}(z,s)$ is $\G$-invariant.
With $m \geq 1$ for the rest of the proof, consider the $\G$-invariant series from \eqref{qmne_1}
\begin{align*}
  Q^{m,m}_\ca(z,s) &:=  \sum_{\g \in \G_\ca\backslash\G} |F_\ca(\g z)|^{2m} \Im(\sai \g z)^s \\
    &=  \sum_{\g \in \G_\ca\backslash\G} \left[(F_\ca( z)+\langle\g,f\rangle)(\overline{F_\ca( z)}+\overline{\langle\g,f\rangle})\right]^m \Im(\sai \g z)^s \\
    &=  E^{m,m}_\ca(z,s) + m\left[\overline{F_\ca( z)}E^{m,m-1}_\ca(z,s)+F_\ca( z)E^{m-1,m}_\ca(z,s) + m\cdot|F_\ca(z)|^2 E_\ca^{m-1,m-1}(z,s)\right]\\ & \quad  + R_\ca^{(m-1)}(z,s).
\end{align*}
The term $R_\ca^{(m-1)}(z,s)$ has a pole at $s=1$ of order at most $m-1$ by part (iii) of Proposition \ref{poles of Em,m-1}.
The series $E_\ca^{m-1,m-1}(z,s)$ possesses a pole of order $m$ at $s=1$ with  lead coefficient $A_{m-1}$ by \eqref{emnpl}. The
series $E^{m,m-1}_\ca(z,s)$ and $E^{m-1,m}_\ca(z,s)$ also have poles of order $m$ at $s=1$ by parts (i) and (ii) of
Proposition \ref{poles of Em,m-1}. Therefore, $Q^{m,m}_\ca(z,s)$ inherits its pole of largest order $m+1$ from $E^{m,m}_\ca$ and it makes sense to set
\begin{equation}\label{K m, a def}
\K^{m,m}_{\ca}(z) := \lim_{s\rightarrow1}
\left(\frac{1}{A_m}(s-1)^m Q^{m,m}_\ca(z,s)-\frac{1}{(s-1)}\right),
\end{equation}
as a natural  analogue of the usual Kronecker limit function.
(As an aside, the  Kronecker limit function
$$
\K^{1,0}_{\ca}(z):=\lim_{s\rightarrow 1}Q^{1,0}_\ca(z,s) = \lim_{s\rightarrow 1}\left(E^{1,0}_\ca(z,s)+F_\ca(z)E_\ca(z,s)\right),
$$
associated to the second-order series $E^{1,0}_\ca(z,s)$, is studied in \cite{JO'S05}.)
 Clearly $\K^{m,m}_{\ca}(z)$ is $\G$-invariant, and
from Proposition \ref{poles of Em,m-1} we have
\begin{align}\notag
  \K^{m,m}_{\ca}(z) &=   \lim_{s\rightarrow1}\left(\frac{1}{A_m}(s-1)^m E^{m,m}_\ca(z,s)-\frac{1}{(s-1)}\right) \\
    & \quad +m \lim_{s\rightarrow1}
(s-1)^m\left[\overline{F_\ca( z)}E^{m,m-1}_\ca(z,s)+F_\ca( z)E^{m-1,m}_\ca(z,s) + \frac{m}{A_m} \cdot|F_\ca(z)|^2 E_\ca^{m-1,m-1}(z,s)\right]\notag\\
      &=  B_\ca^{(m)}(z)-\frac{m^2 A_{m-1}}{A_m}  |F_\ca(z)|^2  = B_\ca^{(m)}(z) -V_f |F_\ca(z)|^2  \label{K m,a limit}.
\end{align}
This proves part (iii) and completes the proof.
\end{proof}

\begin{prop} \label{prop: properties of Ba,m}
Let $m\geq 1$ be an integer.
For any $z=x+iy\in\H$, we can write
\begin{equation}\label{bre}
B_\ca^{(m)}(\sb z)=-\log y+ b_{\ca \cb}^{(m)}(0)+2 \Re (H_{\ca \cb}^{(m)}(z)),
\end{equation}
where $H_{\ca \cb}^{(m)}(z) := \sum_{k> 0} b_{\ca \cb}^{(m)}(k) e(kz)$.
\end{prop}

\begin{proof}
Let us combine the Fourier expansion \eqref{Four. exp E^11} and the Laurent series expansion \eqref{Laur. exp phi^11} together
with
the developments
$$
W_s(z)  =  e(z)+O(s-1)\,\,\,\,\,\textrm{and}\,\,\,\,\,
y^{1-s}  =  1- (s-1)\log y + O((s-1)^2).
$$
This yields
\begin{align*}
B_\ca^{(m)}(\sb z)  &=  \lim_{s\rightarrow1}
\left((s-1)^m \frac{E^{m,m}_\ca(\sb z,s)}{A_m}-\frac{1}{s-1}\right)\\
 &=  \lim_{s\rightarrow 1}\left(\left(\frac{a_{\ca \cb}^{(m)}(0)}{s-1}+b_{\ca \cb}^{(m)}(0)+ \cdots \right)\left(1-(s-1) \log y +  \cdots \right)\right.\\
& \quad \left.+\sum_{k> 0}\left(\frac{a_{\ca \cb}^{(m)}(k)}{s-1}+b_{\ca \cb}^{(m)}(k)+ \cdots \right)\left(e(k z)+ O(s-1)+ \cdots \right)\right.\\
& \quad \left.+\sum_{k<0}\left(\frac{a_{\ca \cb}^{(m)}(k)}{s-1}+b_{\ca \cb}^{(m)}(k)+ \cdots \right)\left(e(k \bar{z})+ O(s-1)+ \cdots \right)-\frac{1}{s-1}\right).
\end{align*}
Since the limit which defines $B_\ca^{(m)}(\sigma_b z)$ exists, it is evident
that we must have $
a_{\ca \cb}^{(m)}(k)=0$ for all $k\neq 0$ and $
a_{\ca \cb}^{(m)}(0)=1$.
Hence we have the expansion
\begin{equation}\label{bexpn}
B_\ca^{(m)}(\sb z)=-\log y+ b_{\ca \cb}^{(m)}(0) +\sum_{k< 0} b_{\ca \cb}^{(m)}(k) e(k \bar{z})+\sum_{k> 0} b_{\ca \cb}^{(m)}(k) e(k z).
\end{equation}
It follows from part (i) of Theorem \ref{bb m} that
\begin{equation}\label{bexpn2}
b_{\ca \cb}^{(m)}(-k)=\overline{b_{\ca \cb}^{(m)}(k)}
\end{equation}
for all $k$ in $\Z$. Therefore, $b_{\ca \cb}^{(m)}(0) \in \R$ and \eqref{bre} holds true, which proves the proposition.
\end{proof}

\begin{prop} \label{up to constant} Let $m\geq 1$ be an integer.
\begin{enumerate}
  \item[(a)] There exists $\mathcal C^{(m)}\in \R$ so that $B_\ca^{(m)}( z) = B_\ca^{(1)}( z) + \mathcal C^{(m)}$.
  \item[(b)] For any pairs of cusps $\ca,$ $\cb$ we have $\mathcal C^{(m)} = b_{\ca \cb}^{(m)}(0)-b_{\ca \cb}^{(1)}(0)$.
  \item[(c)] We have that  $H_{\ca \cb}^{(m)}(z)$ is independent of $m$ and so denote it just by $H^*_{\ca \cb}(z)$.
\end{enumerate}
\end{prop}

\begin{proof}
Consider the difference $B_\ca^{(m)}( z) - B_\ca^{(1)}( z)$. By part (iii) of Theorem \ref{bb m},
this difference is $\G$-invariant. With Proposition \ref{prop: properties of Ba,m} we have
\begin{equation}\label{dcmy}
    B_\ca^{(m)}(\sb z) - B_\ca^{(1)}(\sb z) = b_{\ca \cb}^{(m)}(0) - b_{\ca \cb}^{(1)}(0) +2 \Re \left(H_{\ca \cb}^{(m)}(z) - H_{\ca \cb}^{(1)}(z)\right)
\end{equation}
which combined with Proposition \ref{prop: properties of coeff b ab} shows that $B_\ca^{(m)}( z) - B_\ca^{(1)}( z)$ is bounded in cuspidal zones. Therefore, the difference is
an $L^2$ eigenfunction of the Laplacian with eigenvalue zero by part (ii) of Theorem  \ref{bb m}.
Consequently, the difference is a constant which we denote by $\mathcal C^{(m)}$.
By part (i) of Theorem  \ref{bb m}, $\mathcal C^{(m)}$ must be real which proves part (a).
By letting $z\to i\infty$ in \eqref{dcmy}, parts (b) and (c) follow from (a).
\end{proof}

For the purposes of this article, we do not need the constants $\mathcal C^{(m)}$; however,
it should be possible to determine them exactly. We leave this problem to the
interested reader.

\begin{prop} \label{prop:properties of Hab}
 For any pair of cusps $\ca, \cb$ of $\G$ there exists $S^*_{\ca \cb}: \G \to \R$ such that
\begin{equation}\label{hdiff}
H^*_{\ca \cb}(\g_\cb z)-   H^*_{\ca \cb}(z)
=-\log j(\g_\cb, z)  + V_f \left(
 F_\ca(\sb z)\overline{\langle\g,f\rangle} + |\langle\g,f\rangle|^2/2\right) -2\pi i S^*_{\ca \cb}(\g)
\end{equation}
for all $\g\in\Gamma$ where $\g_\cb := \sb^{-1} \g \sb$.
\end{prop}
\begin{proof}
From part (iii) of Theorem \ref{bb m}, we see that for all $\g$ in $\G$ and all $m \geq 1$
\begin{equation}\label{jhg}
B_\ca^{(m)}(\g z)-V_f|F_\ca(\g z)|^2 -B_\ca^{(m)}(z)+ V_f|F_\ca(z)|^2=0.
\end{equation}
Therefore, Proposition \ref{prop: properties of Ba,m} yields that
\begin{align*}
0  &=  B_\ca^{(m)}(\sb \g_\cb z)-V_f|F_\ca(\g \sb z)|^2 -B_\ca^{(m)}(\sb z)+ V_f|F_\ca(\sb z)|^2\\
  &=  - \log \Im (\g_\cb z) + \log y + 2 \Re \left(H^*_{\ca \cb}(\g_\cb z))- H^*_{\ca \cb}(z)\right)
- V_f \left(|F_\ca(\g \sb z)|^2 - |F_\ca(\sb z)|^2\right) \\
  &=  2 \log |j(\g_\cb, z)|  + 2 \Re \left(H^*_{\ca \cb}(\g_\cb z))- H^*_{\ca \cb}(z)\right)
  -V_f\left(F_\ca(\sb z)\overline{\langle\g,f\rangle}+\overline{F_\ca(\sb z)}\langle\g,f\rangle + |\langle\g,f\rangle|^2\right).
\end{align*}
In other words, we have that $\Re (K^*_{\ca \cb}(\g,z)) =0$
for  the holomorphic function
\begin{equation*} \label{K-ab}
K^*_{\ca \cb}(\g,z)= 2 \log j(\g_\cb, z)  + 2  H^*_{\ca \cb}(\g_\cb z)- 2  H^*_{\ca \cb}(z)
-V_f\left( 2 F_\ca(\sb z)\overline{\langle\g,f\rangle}+ |\langle\g,f\rangle|^2\right).
\end{equation*}
 Consequently, $K^*_{\ca \cb}(\g,z)$ is an imaginary constant which we label as $-4\pi i S^*_{\ca \cb}(\g)$, completing the proof.
\end{proof}

\begin{defs}{\rm
The {\em higher-order modular Dedekind symbol associated to $\G$ and the cusps $\ca$ and $\cb$} is the real-valued function $S^*_{\G, \ca \cb}(\g) = S^*_{\ca \cb}(\g)$ satisfying \eqref{hdiff}. The {\em  higher-order Dedekind eta function associated to $\G$ and the cusps $\ca$ and $\cb$} is
\begin{equation*}
    \eta^*_{\G,\ca\cb}(z) =\eta^*_{\ca\cb}(z):=\exp\left( -\frac{1}{2}H^*_{\ca\cb}(z)+\frac{V_f}{4}|F_\ca(\sigma_\cb z)|^2\right).
\end{equation*}}
\end{defs}

(Of course $S^*_{\ca \cb}$  and $\eta^*_{\ca\cb}$ also depend on the fixed cusp form  $f \in S_2(\G)$.)
Using formulas \eqref{K m, a def}, \eqref{K m,a limit} and \eqref{bre}, we obtain the expression
$$
(s-1)^m \frac{Q_\ca^{m,m}(\sb z,s)}{A_m}=\frac{1}{(s-1)} + b_{\ca\cb}^{(m)}(0)-\log\left( y |\eta^*_{\ca\cb}(z)|^4 \right) + O(s-1), \text{   as   }  s\to1,
$$
which is an analogue of the Kronecker limit formula \eqref{KLT}. In other words, the Kronecker limit function $ \K^{m,m}_{\ca}(\sigma_{\cb}z)$ in \eqref{K m, a def} can be expressed as
$$
 \K^{m,m}_{\ca}(\sigma_{\cb}z)= b_{\ca\cb}^{(m)}(0)-\log\left( y |\eta^*_{\ca\cb}(z)|^4 \right).
$$
 Moreover, Proposition \ref{prop:properties of Hab} implies that
$\log \eta^*_{\ca\cb}(z):=  -H^*_{\ca\cb}(z)/2+V_f|F_\ca(\sigma_\cb z)|^2/4$ satisfies the transformation formula
\begin{equation*} \label{eta ab transf formula}
\log\eta^*_{\ca\cb}(\g_\cb z)=  \log\eta^*_{\ca\cb}(z) + \frac{1}{2} \log j(\g_\cb, z) + \pi i S^*_{\ca\cb}(\g)
\end{equation*}
which has the same form as \eqref{log_eta} and \eqref{etatran}. This justifies the names {\em higher-order modular Dedekind symbol} and {\em higher-order Dedekind eta function}.

\section{Properties of modular Dedekind symbols}

In this section we derive properties of the modular Dedekind symbols $S_{\ca\cb}(\g)$ and $S^*_{\ca\cb}(\g)$. This will complete the proofs of Theorems \ref{main1}, \ref{main2} and Proposition \ref{main3} from the introductory section.  As stated, we will use the abbreviated notation
$S^*_{\ca}$ for $S^*_{\ca\infty}$ and similarly $H^*_{\ca}$ for $H^*_{\ca\infty}$ etc.

\subsection{Basic properties}

\begin{prop} \label{indep}
The symbols $S_{\ca\cb}$ and $S^*_{\ca\cb}$ are independent of the choice of scaling matrices $\sa$, $\sb$. Let the cusp $\cc$ be $\G$-equivalent to $\ca$ with $\cc=\tau \ca$ and $\tau \in \G$. Also let $\cd= \delta\cb$ for  $\delta \in \G$. Then for all $\g \in \G$,
\begin{align}
  S_{\cc\cb}(\g) = S_{\ca\cb}(\g), & \qquad S_{\ca\cd}(\g) = S_{\ca\cb}(\delta^{-1}\g\delta), \label{ind-a}\\
  S^*_{\cc\cb}(\g) = S^*_{\ca\cb}(\g), & \qquad S^*_{\ca\cd}(\g) = S^*_{\ca\cb}(\delta^{-1}\g\delta). \label{ind-b}
\end{align}
\end{prop}
\begin{proof}
We first check that the Eisenstein series $E_\ca^{m,n}(z,s)$ in \eqref{E m,n}, used to define $S_{\ca\cb}$ and $S^*_{\ca\cb}$, are independent of the choice of scaling matrices $\sa$ and the cusp representative $\ca$.

If $\sa$ is a scaling matrix for $\ca$ then, as described in section \ref{scal}, any other scaling matrix $\hat{\sigma}_\ca$ for $\ca$ has the form
$\sa(\ms{1}{t}{0}{1})$ or $\sa(\ms{-1}{-t}{0}{-1})$.  Clearly $\Im(\hat{\sigma}_\ca^{-1}z)=\Im(\sai z -t)=\Im(\sai z)$, showing that $E_\ca^{m,n}(z,s)$ does not depend on the scaling matrix. For $\cc$ in the statement of the proposition, we may take $\sc = \tau \sa$ since $\G_\cc=\tau \G_\ca \tau^{-1}$. Then it is simple to verify, using  $\s{\tau^{-1}\g\tau}{f} = \s{\g}{f}$, that $E_\cc^{m,n}(z,s) = E_\ca^{m,n}(\tau^{-1}z,s) = E_\ca^{m,n}(z,s)$. It follows that $H^{(m)}_{\ca\cb}(z)$  does not depend on $\sa$ and the representative $\ca$. By the defining relations
\eqref{h0} and \eqref{hdiff}, the same is true of $S_{\ca\cb}$ and $S^*_{\ca\cb}$.

We now consider the dependence on the second cusp $\cb$.
Set $\hat{\sigma}_\cb = \sb(\pm \ms{1}{t}{0}{1})$. Let $H^{(0)}_{\ca\cb}(z)$ be the function defined from $B_\ca^{(0)}(\sb z)$ in \eqref{h0expn}  and let $\hat{H}^{(0)}_{\ca\cb}(z)$ be the corresponding function with $\sb$ replaced by $\hat{\sigma}_\cb$. The relation
\begin{equation*}
  \hat{H}^{(0)}_{\ca\cb}(z) = \delta_{\ca\cb}V_\G \cdot it/2+ H^{(0)}_{\ca\cb}(z+t)
\end{equation*}
is easy to obtain. Then, for all $\g \in \G$,
\begin{align*}
  -2\pi i \hat{S}_{\ca\cb}(\g) & := \log j(\hat{\sigma}_\cb^{-1} \g \hat{\sigma}_\cb, z) +
  \hat{H}^{(0)}_{\ca\cb}(\hat{\sigma}_\cb^{-1} \g \hat{\sigma}_\cb z)- \hat{H}^{(0)}_{\ca\cb}(z) \\
   & = \log j(\g_\cb, z+t) + H^{(0)}_{\ca\cb}(\g_\cb(z+t))- H^{(0)}_{\ca\cb}(z+t)
    = -2\pi i S_{\ca\cb}(\g).
\end{align*}
Hence $S_{\ca\cb}$ does not depend on the choice of $\sb$. A similar argument shows the same for $S^*_{\ca\cb}$.

It remains to prove the right equalities in \eqref{ind-a}, \eqref{ind-b}. For $\cd=\delta \cb$ we may take  $\sd = \delta \sb$. Then $B_\ca^{(0)}(\sd z) = B_\ca^{(0)}(\delta \sb z) = B_\ca^{(0)}(\sb z)$. This means that $H^{(0)}_{\ca\cd}(z) = H^{(0)}_{\ca\cb}(z)$. With the defining relation
\eqref{h0} we obtain
\begin{align*}
  -2\pi i S_{\ca\cd}(\g) & =  \log j(\g_\cd, z) + H^{(0)}_{\ca\cd}(\g_\cd z))- H^{(0)}_{\ca\cd}(z) \\
   & =  \log j(\sbi(\delta^{-1}\g \delta)\sb, z) + H^{(0)}_{\ca\cb}(\sbi(\delta^{-1}\g \delta)\sb z))- H^{(0)}_{\ca\cb}(z) \\
   &= -2\pi i S_{\ca\cb}(\delta^{-1}\g \delta).
\end{align*}
For $m\geq 1$,  $B_\ca^{(m)}(z)$ is no longer  $\G$-invariant. With part (iii) of Theorem \ref{bb m} we find
\begin{equation*}
  B_\ca^{(m)}(\sd z) = B_\ca^{(m)}(\delta \sb z) = B_\ca^{(m)}(\sb z) +V_f|\s{\delta}{f}|^2
  +2V_f \Re\left(F_\ca(\sb z)\overline{\s{\delta}{f}} \right)
\end{equation*}
and so
\begin{equation} \label{hadz}
  H_{\ca\cd}^{*}( z) = H_{\ca\cb}^{*}( z) + V_f F_\ca(\sb z)\overline{\s{\delta}{f}}.
\end{equation}
The right equality in  \eqref{ind-b} now follows from \eqref{hdiff} and \eqref{hadz}.
\end{proof}

The next result shows that    $S_{\ca\cb}$ and $S_\ca$, as well as $S^*_{\ca\cb}$ and $S^*_\ca$, always agree up to an additive factor  from the set $\{-2,-1,0,1,2\}$. This reduces the study of $S_{\ca\cb}$ and $S^*_{\ca\cb}$ to $S_\ca$ and $S^*_\ca$, respectively.

\begin{prop} \label{31} For all $\g$  in $\G$ and any pair of cusps $\ca$ and $\cb$ we have
\begin{align*}
  S_{\ca\cb}(\g ) & =S_{\ca} (\g) +
  \omega (\sbi , \g )-   \omega(\g_{\cb}, \sbi), \\
  S^*_{\ca\cb}(\g ) & =S^*_{\ca} (\g) +
  \omega (\sbi , \g )-   \omega(\g_{\cb}, \sbi)
\end{align*}
where $\g_\cb := \sb^{-1} \g \sb$.
\end{prop}
\begin{proof}
By \eqref{abh} and Proposition \ref{prop: properties of Ba,m} we can write
\begin{align*}
B_\ca ^{(m)}(z)  &=  -\log y+ b_{\ca \infty}^{(m)}(0)+2 \Re (H^{(m)}_{\ca} (z)),\\
B_\ca ^{(m)}(z)=B_\ca ^{(m)}(\sb(\sbi z))  &=  -\log \Im (\sbi z)+ b_{\ca \cb}^{(m)} (0)+2 \Re (H^{(m)}_{\ca \cb}(\sbi z)).
\end{align*}
for $m \geq 0$. Hence
$$
 b_{\ca \infty} ^{(m)}(0)+2 \Re (H^{(m)}_{\ca}(z))=2 \log |j(\sbi, z)|+ b_{\ca \cb}^{(m)}(0)+2 \Re (H^{(m)}_{\ca \cb}(\sbi z))
$$
and $\Re (L_{\ca \cb}^{(m)}(z))=0$ for the holomorphic function
$$
L_{\ca \cb}^{(m)}(z) := b_{\ca \infty}^{(m)}(0)- b_{\ca \cb}^{(m)}(0) +2 H^{(m)}_{\ca}(z)-2 H^{(m)}_{\ca \cb}(\sbi z) -2\log j(\sbi, z).
$$
It follows that $L_{\ca \cb}^{(m)}(z)=2i k_{\ca \cb}^{(m)}$ is an imaginary constant independent of $z \in \H$ and
$$
H^{(m)}_{\ca }(z)= H^{(m)}_{\ca \cb}(\sbi z)+ (b_{\ca \cb}^{(m)}(0)-b_{\ca \infty}^{(m)}(0))/2+i k_{\ca \cb}^{(m)} + \log j(\sbi, z).
$$
Thus
\begin{equation}\label{typ}
H^{(m)}_{\ca }(\g z)-  H^{(m)}_{\ca }(z)= H^{(m)}_{\ca \cb}(\g_{\cb} \sbi z) - H^{(m)}_{\ca \cb}(\sbi z)+ \log j(\sbi, \g z)- \log j(\sbi, z).
\end{equation}
 We apply  \eqref{h0} to both sides of \eqref{typ} in the case when $m=0$ to get
\begin{align*}
-2\pi i\bigl( S_{\ca}(\g)- S_{\ca \cb}(\g)\bigr)
 &=  \log j(\g, z)+\log j(\sbi, \g z)- \log j(\sbi, z)
- \log j(\g_\cb, \sbi z) \\
 &=  2\pi i \bigl(\omega (\sbi , \g)-  \omega(\g_\cb, \sbi) \bigr) .
\end{align*}
The argument is the same when $m\geq 1$ using (\ref{hdiff}) in  \eqref{typ}.
This completes the proof.
\end{proof}

\subsection{The transformation properties of $S^*_\ca$}

In this subsection we prove part (i) of Theorem \ref{main2} and equation \eqref{mn-s3} of Proposition \ref{main3}.

\begin{prop} \label{32} For all $\g$ and $\tau$ in $\G$ and any cusp $\ca$,
$$
S^*_\ca (\g \tau) =S^*_\ca(\g) +S^*_\ca(\tau) + \frac{V_f}{2\pi} \Im \left(\langle\g,f\rangle \overline{\langle\tau,f\rangle}\right)  + \omega(\g,\tau).
$$
\end{prop}
\begin{proof} Replace $z$ by $\tau z$ in (\ref{hdiff}) to get
\begin{equation}\label{h7}
H^*_\ca(\g \tau z)-H^*_\ca(\tau z)=- \log j(\g, \tau z)  +
 V_f F_\ca(\tau z)\overline{\langle\g,f\rangle} + V_f |\langle\g,f\rangle|^2/2 -2\pi i S^*_\ca(\g).
\end{equation}
Replace $\g$ by $\tau$ in (\ref{hdiff}) and add to (\ref{h7}) to get
\begin{align}
H^*_\ca(\g \tau z)-H^*_\ca( z)   &= - \log j(\g, \tau z)  +
 V_f F_\ca(\tau z)\overline{\langle\g,f\rangle} + V_f |\langle\g,f\rangle|^2/2 -2\pi i S^*_\ca(\g) \nonumber\\
 & \quad - \log j(\tau, z)  +
 V_f F_\ca(z)\overline{\langle\tau,f\rangle} + V_f |\langle\tau,f\rangle|^2/2 -2\pi i S^*_\ca(\tau). \label{35}
\end{align}
However, replacing $\g$ by $\g\tau$ in (\ref{hdiff}) also yields
\begin{equation}\label{36}
H^*_\ca(\g \tau z)-H^*_\ca(z) =- \log j(\g\tau, z)  +
 V_f F_\ca(z)\overline{\langle\g\tau,f\rangle} + V_f |\langle\g\tau,f\rangle|^2/2 -2\pi i S^*_\ca(\g\tau).
\end{equation}
Equating (\ref{35}) and (\ref{36}) gives
\begin{align*}
\lefteqn{
 -2\pi i\bigl( S^*_\ca(\g\tau)-  S^*_\ca(\g)- S^*_\ca(\tau)\bigr)- \log j(\g\tau, z) + \log j(\g, \tau z)+ \log j(\tau, z)}\hspace{30mm} \\
   &=  -2\pi i\bigl( S^*_\ca(\g\tau)-  S^*_\ca(\g)- S^*_\ca(\tau)\bigr)+2\pi i \omega(\g,\tau)\\
   &=  - V_f F_\ca(z) \left(\overline{\langle\g,f\rangle+\langle\tau,f\rangle} \right)  +
 V_f \Bigl(F_\ca( z)+\langle\tau,f\rangle \Bigr)\overline{\langle\g,f\rangle} +
 V_f F_\ca(z)\overline{\langle\tau,f\rangle} \\
 & \quad + V_f |\langle\g,f\rangle|^2/2 + V_f |\langle\tau,f\rangle|^2/2 - V_f \Bigl|\langle\g,f\rangle+\langle\tau,f\rangle \Bigr|^2/2\\
   &=  -i V_f \Im \left(\langle\g,f\rangle \overline{\langle\tau,f\rangle} \right),
\end{align*}
which completes the proof.
\end{proof}

\begin{cor} \label{prop:special prop. of theta}
For every $\g$ in $\G$ and any cusp $\ca$, we have the following properties of $S^*_{\ca}$:
\begin{enumerate}
\item $S^*_\ca (I)=0$;
\item $S^*_\ca (-I)=-1/2$ \quad if  \quad $-I \in \G$;
\item $S^*_\ca (-\g)=S^*_\ca( \g)- 1/2 + \omega(-I,\g)$ \quad  if  \quad $-I \in \G$;
\item $S^*_\ca (\g^{-1})=-S^*_\ca (\g)- \omega( \g, \g^{-1})$.
\end{enumerate}
\end{cor}
\begin{proof} These assertions all follow from Proposition \ref{32} and equation (\ref{w2}).
\end{proof}

\begin{prop} \label{cor33} For all $\g_1$, $\g_2$ and $\g_3$ in $\G$ and any cusp $\ca$,
\begin{gather*}
S^*_\ca (\g_1 \g_2 \g_3 ) -S^*_\ca (\g_1 \g_2  )-S^*_\ca (\g_1  \g_3 )-S^*_\ca ( \g_2 \g_3 )+S^*_\ca (\g_1)+S^*_\ca (\g_2)
+S^*_\ca (\g_3)\\  =
 \omega(\g_1 \g_2,\g_3) -\omega(\g_1,\g_3)-\omega(\g_2,\g_3).
\end{gather*}
\end{prop}

\begin{proof}
From Proposition \ref{32} we have that
\begin{align*}
S^*_\ca ((\g_1 \g_2) \g_3)  &=  S^*_\ca (\g_1 \g_2) +S^*_\ca (\g_3) + \frac{V_f}{2\pi} \Im \left(\langle\g_1 \g_2,f\rangle \overline{\langle\g_3,f\rangle}\right)
  + \omega(\g_1 \g_2,\g_3)\\
 &=  S^*_\ca (\g_1 \g_2) +S^*_\ca (\g_3) + \frac{V_f}{2\pi} \Im \left(\bigl(\langle\g_1,f\rangle +\langle\g_2,f\rangle \bigr)\overline{\langle\g_3,f\rangle}\right)
 + \omega(\g_1 \g_2,\g_3).
\end{align*}
By combining with the formulas
$$
\frac{V_f}{2\pi} \Im \left(\langle\g_1,f\rangle \overline{\langle\g_3,f\rangle}\right)  =
S^*_\ca (\g_1) +S^*_\ca (\g_3)-S^*_\ca (\g_1 \g_3)  -  \omega(\g_1,\g_3)
$$
and
$$
\frac{V_f}{2\pi} \Im \left(\langle\g_2,f\rangle \overline{\langle\g_3,f\rangle}\right)  =
S^*_\ca (\g_2) +S^*_\ca (\g_3)-S^*_\ca (\g_2 \g_3)  -  \omega(\g_2,\g_3)
$$
the proof follows.
\end{proof}

\subsection{Evaluation for parabolic and elliptic elements of $\G$} \label{ellpr}

\begin{prop} \label{prop35f} Let $\g$ be a parabolic element of $\G$, fixing the cusp $\cb$. As usual, we may choose a scaling matrix $\sb$ so that $\g_\cb:=\sbi\g \sb= \pm(\smallmatrix 1
& h \\ 0 & 1 \endsmallmatrix )$ for $h\in \Z$. Then, for any cusp $\ca$,
\begin{equation}\label{par-sf}
S_\ca (\g)= \frac{\delta_{\ca\cb} V_\G \cdot h}{4\pi} -\frac{\log(j(\g_\cb,i))}{2\pi i}-  \omega(\sbi,\g)+ \omega(\g_\cb,\sbi).
\end{equation}
In the case that $\g_\cb=(\smallmatrix 1
& h \\ 0 & 1 \endsmallmatrix )$, \eqref{par-sf} simplifies to
\begin{equation*}\label{par-s2f}
S_\ca (\g)=  \delta_{\ca\cb} V_\G \cdot h/(4\pi).
\end{equation*}
\end{prop}

\begin{proof}  Equation (\ref{h0}) implies that
\begin{equation*}
    H^{(0)}_{\ca \cb}(z+h)-   H^{(0)}_{\ca \cb}(z)
=-\log j(\g_\cb, z)   -2\pi i S_{\ca \cb}(\g)
\end{equation*}
and therefore $2\pi i S_{\ca \cb}(\g) = \delta_{\ca\cb}V_\G \cdot ih/2 -\log j(\g_\cb, z)$ by \eqref{h0expn}. Of course $j(\g_\cb, z)$ is independent of $z$ and equals $\pm 1$.
Proposition \ref{31}, relating $S_\ca$ to $S_{\ca \cb}$ now proves \eqref{par-sf}.

Suppose $\g_\cb=(\smallmatrix 1
& h \\ 0 & 1 \endsmallmatrix )$. Then $j(\g_\cb, z)=1$ and the second term on the right of \eqref{par-sf} is $0$. The third term is $-\omega(\sbi,\sb (\smallmatrix 1
& h \\ 0 & 1 \endsmallmatrix )\sbi)$ and equals $0$ by an application of Proposition \ref{cases} or \cite[Eq. 2.47]{Iw2}. Lastly, the fourth term $w((\smallmatrix 1
& h \\ 0 & 1 \endsmallmatrix ),\sbi)$ is $0$ as a result of \eqref{w2}.
\end{proof}

The same proof, using Propositions \ref{prop: properties of Ba,m} and \ref{prop:properties of Hab} for $ H^{(m)}_{\ca \cb}$ shows the higher-order result:

\begin{prop} \label{prop35} Let $\g$ be a parabolic element of $\G$, fixing the cusp $\cb$. Choose $\sb$ so that $\g_\cb:=\sbi\g \sb= \pm(\smallmatrix 1
& h \\ 0 & 1 \endsmallmatrix )$ for $h\in \Z$. Then, for any cusp $\ca$,
\begin{equation}\label{par-s}
S^*_\ca (\g)=-\frac{\log(j(\g_\cb,i))}{2\pi i}-  \omega(\sbi,\g)+ \omega(\g_\cb,\sbi).
\end{equation}
In the case that $\g_\cb=(\smallmatrix 1
& h \\ 0 & 1 \endsmallmatrix )$, \eqref{par-s} simplifies to $S^*_\ca (\g)=0$.
\end{prop}

\begin{prop} \label{prop36} Let $\ca$ be any cusp for $\G$ and let $E \in \G$ be an elliptic element. Suppose $E^r=I$ for $r>0$ and that $E$  fixes $z_0 \in \H$. Then $S_\ca (E)  = S^*_\ca (E)$ and this common value  is given by
\begin{equation} \label{lftrgt}
   -\frac 1{2\pi i} \log j( E, z_0) \quad \text{and} \quad  -\frac{1}{r} \sum_{k=1}^{r-1}
\omega(E^k, E ).
\end{equation}
\end{prop}
\begin{proof} The formula on the left of \eqref{lftrgt} equals  $S_\ca (E)$ by \eqref{h0} and equals $S^*_\ca (E)$ by  (\ref{hdiff}) since $\s{E}{f}=0$. The formula on the right of \eqref{lftrgt} follows from \eqref{ee} and Proposition \ref{32}.
\end{proof}

\begin{cor} \label{par_ell_rat}
The modular Dedekind symbols $S_{\ca \cb}(\g)$ and $S^*_{\ca \cb}(\g)$ are always rational for all parabolic and elliptic elements $\g$ of $\G$.
\end{cor}
\begin{proof}
With Proposition \ref{31} we just need to demonstrate that  $S_{\ca}(\g)$, $S^*_{\ca}(\g) \in \Q$. This follows from Propostions \ref{prop35f} - \ref{prop36} along with \eqref{gb}.
\end{proof}

Corollary \ref{par_ell_rat} extends Corollary 4.2 and Theorem 4.5 of \cite{Gn1} where Goldstein showed that $S_{\ca \ca}(\g)$ is always rational for  elliptic and parabolic $\g$.

Let $\ca$ and $\cb$ be any two cusps for $\G$. A consequence of Proposition \ref{prop35f}  is that $S_\ca(\g)$ may not equal $S_\cb(\g)$ for $\g$ parabolic. However, Propositions \ref{prop35} and \ref{prop36} show that $S_\ca(\g)=S_\cb(\g)$  if $\g$ is elliptic and $S^*_\ca(\g)=S^*_\cb(\g)$ if $\g$ is parabolic or elliptic. The  results in the next section show the general relation between modular Dedekind symbols at different cusps.

\subsection{The relations between $S_\ca$ and $S_\cb$ and between  $S^*_\ca$ and $S^*_\cb$}
In the following proofs we extend the usual definition of  $\frac{d}{dz}$ on holomorphic functions to $\frac{1}{2}(\frac{d}{dx} - i\frac{d}{dy})$, the first-order Wirtinger derivative. This derivative applies to any real-analytic function, for example.

\begin{prop} \label{ea2}
The limit \eqref{heck} defining $E_{\ca,2}(z)$ exists and $E_{\ca,2}(z) + 1/(V_\G \cdot y)$ is a holomorphic function of $z\in \H$. For any two cusps, $\ca$ and $\cb$, the difference $E_{\ca,2}(z) - E_{\cb,2}(z)$ is a holomorphic modular form of weight $2$ for $\G$.
\end{prop}
\begin{proof}
We first show that \eqref{heck} is equivalent to another limit:
\begin{align*}
  2i  \lim_{s\to 1^+} \frac{d}{dz} E_\ca(z,s) & = 2i  \lim_{s\to 1^+}  \sum_{\g \in \G_\ca\backslash\G}\frac{d}{dz} \Im(\sai\g z)^s \\
   &  = 2i  \lim_{s\to 1^+}  \sum_{\g \in \G_\ca\backslash\G} \frac{-i s}{2}  j(\sai\g, z)^{-2} \Im(\sai\g z)^{s-1} \\
   &  = 2i  \lim_{s\to 0^+}  (s+1) \sum_{\g \in \G_\ca\backslash\G} j(\sai\g, z)^{-2} \Im(\sai\g z)^{s}.
\end{align*}
Therefore, the limit $E_{\ca,2}(z)$ exists since, with \eqref{dhk},
\begin{equation} \label{inx}
  V_\G E_{\ca,2}(z) = 2i  V_\G \lim_{s\to 1^+} \frac{d}{dz} E_\ca(z,s) = 2i  \frac{d}{dz} B_\ca^{(0)}(z).
\end{equation}
Clearly $E_{\ca,2}(z)$ has weight $2$. Using \eqref{bexpn0} in \eqref{inx}, the Fourier expansion at the cusp $\cb$ is
\begin{equation*}
  j(\sb,z)^{-2} E_{\ca,2}(\sb z) = -\frac{1}{V_\G \cdot y}+ \delta_{\ca \cb} - \frac{4\pi}{V_\G} \sum_{k>0}k\cdot b^{(0)}_{\ca \cb}(k)e(k z).
\end{equation*}
It follows that $E_{\ca,2}(z) + 1/(V_\G \cdot y)$ is  holomorphic. The expansion of $E_{\ca,2}(z) - E_{\cb,2}(z)$ at the cusp $\cc$ is
\begin{equation*}
  j(\sc,z)^{-2} \left(E_{\ca,2}(\sc z)-E_{\cb,2}(\sc z) \right) =  \delta_{\ca \cc} - \delta_{\cb \cc} - \frac{4\pi}{V_\G} \sum_{k>0}k \left( b^{(0)}_{\ca \cc}(k) - b^{(0)}_{\cb \cc}(k) \right) e(k z)
\end{equation*}
indicating that this difference is a holomorphic modular form and not necessarily a cusp form.
\end{proof}

\begin{prop} \label{prop deriv b} For every pair of cusps $\ca$, $\cb$ and all $m \geq 0$,
\begin{equation}\label{ga1}
    \frac{d}{dz}\left( B_\ca^{(m)}(z) - B_\cb^{(m)}(z)\right) =
    \begin{cases}
      V_\G\bigl(E_{\ca,2}(z) - E_{\cb,2}(z)\bigr)/(2i) & \mbox{if \ } m=0 \\
      2\pi i V_f \overline{F_{\ca}( \cb)} \cdot f(z) & \mbox{if \ } m \geq 1.
    \end{cases}
\end{equation}
Also
\begin{align} \label{ga2v}
    H^{(0)}_\ca(\g z) - H^{(0)}_\ca( z) -\bigl(H^{(0)}_\cb(\g z) - H^{(0)}_\cb( z)\bigr) & = -V_\G \langle \g,E_{\ca,2} - E_{\cb,2}\rangle/(4\pi),\\
    \label{ga2}
    H^*_\ca(\g z) - H^*_\ca( z) -\bigl(H^*_\cb(\g z) - H^*_\cb( z)\bigr) & = V_f  \overline{F_{\ca}( \cb)} \langle \g,f\rangle.
\end{align}
\end{prop}

\begin{proof}
Set
\begin{equation*}
   g^{(m)}_{\ca \cb}(z) := \frac{1}{2\pi i} \frac{d}{dz}\left( B_\ca^{(m)}(z) - B_\cb^{(m)}(z)\right).
\end{equation*}
We have already seen in Proposition \ref{ea2} that $g^{(0)}_{\ca \cb}(z)$ is a holomorphic modular form of weight $2$ for $\G$, equaling $ -V_\G\bigl(E_{\ca,2}(z) - E_{\cb,2}(z)\bigr)/(4\pi)$.
 Acting by $\frac{d}{dz}$ on \eqref{bexpn} and \eqref{jhg}, we can similarly show that $g^{(m)}_{\ca \cb}(z)$ is  in  $S_2(\G)$ for $m\geq 1$.

We next compute the period
\begin{align}
\langle \g,g^{(m)}_{\ca \cb} \rangle &=  \int_z^{\g z}  \Bigl( -\delta_{0m}(\delta_{\ca\ci}-\delta_{\cb\ci}) i V_\G/2 + 2\pi i \sum_{k>0} k \bigl(b_{\ca \ci}(k)-b_{\cb \ci}(k)\bigr) e(k w) \Bigr) \, dw \notag\\
 &=  \Bigl(  -\delta_{0m}(\delta_{\ca\ci}-\delta_{\cb\ci}) i V_\G\cdot z/2 + \sum_{k>0} \bigl(b_{\ca \ci}(k)-b_{\cb \ci}(k) \bigr) e(k w) \Bigr)\Big]_z^{\g z}  \notag\\
 &=  H^{(m)}_\ca(\g z) - H^{(m)}_\ca( z) -\bigl(H^{(m)}_\cb(\g z) - H^{(m)}_\cb( z)\bigr). \label{ga3}
\end{align}
Equation \eqref{ga2v} follows from \eqref{ga3} and \eqref{ga1} when $m=0$.
Taking the real parts of \eqref{ga3} for $m\geq 1$ gives
\begin{align*}
2\Re \langle \g,g^{(m)}_{\ca \cb} \rangle  &=  2\Re (H^*_\ca(\g z) - H^*_\ca( z)) - 2\Re (H^*_\cb(\g z) - H^*_\cb( z)) \\
 &=   B^{(m)}_\ca(\g z) - B^{(m)}_\ca( z) -  (B^{(m)}_\cb(\g z) - B^{(m)}_\cb( z)) \\
 &=  V_f \left(|F_\ca(\g z)|^2- |F_\ca( z)|^2-|F_\cb(\g z)|^2+ |F_\cb( z)|^2 \right)\\
 &=  V_f \left(\langle \g,f \rangle \overline{F_{\ca}( \cb)}+ \overline{\langle \g,f \rangle}F_{\ca}( \cb) \right) \\
 &=  2\Re \left( V_f \overline{F_{\ca}( \cb)} \langle \g,f \rangle  \right)
\end{align*}
where we used \eqref{jhg}. Hence, for all $\g \in \G$,
\begin{equation*}
\Re \left\langle \g,g^{(m)}_{\ca \cb} - V_f \overline{F_{\ca}( \cb)} \cdot f \right\rangle =0.
\end{equation*}

We may now employ a formula of Zagier, \cite[Thm. 1]{Za85a}. It states that, for any $g\in S_2(\G)$,
\begin{equation}\label{zagier}
  ||g||^2 =\frac{1}{8\pi^2} \sum \Im\left(\s{\g_i}{g} \overline{\s{\g_{j}}{g}} \right)
\end{equation}
where the finite  sum is over  certain pairs  $\g_i$, $\g_{j}$ from a set of generators for $\G$. In particular, if the real parts of all the modular symbols $\s{\g}{g}$ are zero then we deduce that the Petersson norm of $g$ is zero.

Applying this argument to $g:=g^{(m)}_{\ca \cb} - V_f \overline{F_{\ca}( \cb)} \cdot f$  demonstrates that $g^{(m)}_{\ca \cb} = V_f \overline{F_{\ca}( \cb)} \cdot f $ for $ m\geq 1$.
This equality implies \eqref{ga1} and it also implies  \eqref{ga2} using \eqref{ga3}.
\end{proof}

\begin{cor} \label{prop relating a and b}For every pair of cusps $\ca$, $\cb$ and all $\g$ in $\G$  we have
\begin{align}\label{ssg}
  S_\ca (\g)  & = S_\cb (\g) +\frac{V_\G}{8\pi^2 i}  \langle \g,E_{\ca,2} - E_{\cb,2} \rangle, \\
  S^*_\ca (\g) &  = S^*_\cb (\g) -\frac{V_f}{\pi} \Im \left(\overline{F_{\ca}( \cb)} \cdot \langle \g,f \rangle \right). \label{ssg2}
\end{align}
\end{cor}
\begin{proof}
By \eqref{ga2} and \eqref{hdiff}
\begin{align*}
 V_f  \overline{F_{\ca}( \cb)} \langle \g,  f\rangle  &=  H^*_\ca(\g z) - H^*_\ca( z) -\bigl(H^*_\cb(\g z) - H^*_\cb( z)\bigr) \\
 &=  V_f F_\ca(z) \overline{\langle \g,f \rangle} - V_f F_\cb(z) \overline{\langle \g,f \rangle}
-2\pi i \bigl(S^*_\ca (\g) - S^*_\cb(\g)\bigr)\\
 &=  V_f F_{\ca}( \cb)  \overline{\langle \g,f \rangle} -2\pi i \bigl(S^*_\ca (\g) - S^*_\cb(\g)\bigr).
\end{align*}
Equation \eqref{ssg2} follows. The proof of  \eqref{ssg} is similar, using \eqref{ga2v} and \eqref{h0}.
\end{proof}

\subsection{The difference $S^*_\ca - S_\ca$} \label{pss}

For any cusp $\ca$ of $\Gamma$, define the function
\begin{equation}\label{theta-def}
\theta_\ca := S^*_\ca - S_\ca
\end{equation}
which is the difference between our two types of  modular Dedekind symbols.
This difference is particularly nice since the phase factors drop out. Also note that $S^*_{\ca\cb} - S_{\ca\cb}$ is independent of $\cb$ by Proposition \ref{31}, so the definition \eqref{theta-def} does not lose any generality.

 For example, we have the following relations. Recall that $V_f$ is defined by \eqref{V f def}.

\begin{prop} \label{cor:rpoperties of Psi} For every $\g$ and $\tau\in\Gamma$ we have
 \begin{equation}\label{53}
\theta_\ca (\g \tau)= \theta_\ca (\g) +\theta_\ca (\tau) + \frac{V_f}{2\pi} \Im \left(\langle\g,f\rangle \overline{\langle\tau,f\rangle}\right).
\end{equation}
Also
\begin{align}
\theta_\ca (I) &= 0, \label{thetax1}\\
\theta_\ca (-\g) &=  \theta_\ca (\g) \quad \text{if} \quad -I\in\Gamma, \label{thetax2}\\
\theta_\ca (\g^{-1}) &= -\theta_\ca (\g). \label{thetax3}
\end{align}
Moreover, $\theta_\ca(E)=0$ for every elliptic element $E\in\G$ and for a parabolic element $\g\in\G$, fixing the cusp $\cb$ with a scaling matrix $\sb$ so that $\g_\cb:=\sbi\g \sb= \pm(\smallmatrix 1
& h \\ 0 & 1 \endsmallmatrix )$ for $h\in \Z$, one has $$\theta_\ca (\g)= -\frac{\delta_{\ca\cb} V_\G \cdot h}{4\pi}.$$
\end{prop}

\begin{proof}
Equation \eqref{53} follows directly from \eqref{ee} and Proposition \ref{32}. The next three identities \eqref{thetax1} - \eqref{thetax3}  follow directly from \eqref{53}. For the last identity, \eqref{thetax3}, note that $\langle \g^{-1},f\rangle = -\langle \g,f\rangle$, so then
$$
\Im (\langle\g,f\rangle \overline{\langle\g^{-1},f\rangle})=-\Im (\langle\g,f\rangle \overline{\langle\g,f\rangle})=-\Im (|\langle\g,f\rangle|^2)=0.
$$
The last two statements on elliptic and parabolic elements follow immediately from Proposition \ref{prop36} and Propositions \ref{prop35f} and \ref{prop35}.
\end{proof}

\begin{prop} \label{prop514}
For any three elements $\g_1,\g_2,\g_3 \in\Gamma$  we have
\begin{equation*}\label{5iv}
\theta_\ca (\g_1 \g_2 \g_3 ) =\theta_\ca (\g_1 \g_2  )+\theta_\ca (\g_1  \g_3 )+\theta_\ca ( \g_2 \g_3 )-\theta_\ca (\g_1)-\theta_\ca (\g_2)
-\theta_\ca (\g_3)
\end{equation*}
\end{prop}
\begin{proof}
An immediate consequence of Proposition \ref{cor33} and formula \eqref{triple prod. for xi}.
\end{proof}

Proposition \ref{prop514} completes the proof of Proposition \ref{main3}.

\section{Symmetries and rationality} \label{symd}

In this section we consider how the modular Dedekind symbols behave when combined with several interesting operators.
First, we consider Fuchsian groups $\Gamma$ which are not necessarily arithmetic
yet are invariant with respect to the operator $\iota$ defined by

\begin{equation*} \label{iota def}
\left(
                              \begin{array}{cc}
                                a & b \\
                                c & d \\
                              \end{array}
                            \right) \stackrel{\iota}{\longrightarrow}
\left(
                              \begin{array}{cc}
                                a & -b \\
                                -c & d \\
                              \end{array}
                            \right).
\end{equation*}
This operator is used in \cite[Lemma 6]{Asai} and in \cite[Sect. 3]{Go90} where it is shown to have a simple effect on modular Dedekind symbols and modular symbols, respectively.

Secondly, we study the Fricke involution
$$
\left(
                              \begin{array}{cc}
                                a & b \\
                                c & d \\
                              \end{array}
                            \right) \stackrel{w_N}{\longrightarrow}
\left(
                              \begin{array}{cc}
                                d & -c/N \\
                                -bN & a \\
                              \end{array}
                            \right)
$$
and its generalization, the Atkin-Lehner involution,
which are both defined on the congruence group $\Gamma_0(N)$.

These operators allow us to prove the rationality of  higher-order modular Dedekind symbols on certain genus one surfaces $\G_0(N)\backslash \H$ in section \ref{rat_thm_prf}.

\subsection{Composition with the involution $\iota$}

We start with a lemma which summarizes the properties of the operator $\iota$. Part (iii) appears in \cite[Prop. 2]{Go90}.

\begin{lemma} \label{iotalem}
The following statements hold:
\begin{enumerate}
\item For all $\g$ in $\slr$ we have \quad $\g(-\overline z)=-\overline{(\iota (\g) z)}$.
\item Also, for all $\g$, $\tau$ in $\slr$ \quad $\iota(\g \tau)=\iota(\g )\iota(\tau)$.
\item Let $\Gamma$ be a Fuchsian group invariant under the action of $\iota$ and with a cusp at $\ci$ and scaling matrix $\si=I$ as usual. Suppose $f \in S_2(\G)$ has real Fourier coefficients in its expansion at $\ci$. Then for all $\g \in \G$
$$
\langle \iota(\g),f \rangle = \overline{\langle \g,f \rangle}.
$$
\end{enumerate}
\end{lemma}

\begin{proof}
The proofs of the first two statements are straightforward.
With the notation \eqref{jff} we have $F_\ci(z) =   \sum_{n=1}^\infty (a_\ci(n)/n) \cdot
e(nz)$ by \eqref{jff2}. Therefore $F_\ci(-\overline z)= \overline{F_\ci( z)}$ and hence
\begin{align*}
    \langle \iota(\g),f \rangle & = F_\ci\bigl(\iota(\g)(-\overline z)\bigr) - F_\ci(-\overline z)\\
    & = F_\ci\bigl(-\overline{(\g z)}\bigr) - F_\ci(-\overline z)\\
     & = \overline{F_\ci( \g z)} - \overline{F_\ci( z)} = \overline{\langle \g,f \rangle}. \qedhere
\end{align*}
\end{proof}

Recall the notation $\rho(\g)$ from \eqref{rho-def}.

\begin{prop} \label{prop:43} Let $\Gamma$ be a Fuchsian group,  such that $\iota(\g) \in \Gamma$ for all $\g\in\Gamma$, and with a cusp at $\ci$ and scaling matrix $\si=I$.
Assume that $f\in S_2(\Gamma)$ has real Fourier coefficients in its expansion at $\ci$. Then for every $\g$ in $\Gamma$
$$
S^*_\ci(\iota(\g))=-S^*_\ci(\g) - \rho(\g).
$$
\end{prop}
\begin{proof}

Directly, one has the calculation that
\begin{align*}
E^{m,m}_\ci(z,s) &=\sum_{\g \in \G_\infty\backslash\G} |\langle\g,f \rangle|^{2m} \Im(\g z)^s
=\sum_{\g \in \G_\infty\backslash\G}|\langle\iota(\g),f \rangle|^{2m} \Im(\iota(\g) z)^s\\
&=\sum_{\g \in \G_\infty\backslash\G} |\langle\g,f \rangle|^{2m} \Im(\g (-\overline{z}))^s
=E^{m,m}_\ci(-\overline{z},s)
\end{align*}
for all integers $m \geq 0$.
It follows that $
B^{(m)}_\ci( -\overline{z})=B^{(m)}_\ci ( z)$, and hence, from \eqref{bexpn}, we conclude that
$b^{(m)}_{\ci\ci}(-k)=b^{(m)}_{\ci\ci}(k)$
for all $k \in \Z$. Since we have shown that $b^{(m)}_{\ci\ci}(-k)=\overline{b^{(m)}_{\ci\ci}(k)}$ by \eqref{bexpn2},
we see that $b^{(m)}_{\ci\ci}(k) \in \R$ and, furthermore,
$$
H^{(m)}_\ci(-\overline{z})=\overline{H^{(m)}_\ci(z)}.
$$
Therefore, beginning with \eqref{hdiff} and applying  part (i) of Lemma \ref{lemma: preparation for iota prop}, we have that
\begin{align*}
-2\pi i S^*_\ci(\iota(\g))  &=  \log j(\iota(\g),z) +H^*_\ci(\iota(\g) z)-H^*_\ci(z)- V_f \left(F_\ci(z) \overline{\langle \iota(\g),f\rangle}
+ | \langle \iota(\g),f \rangle|^2/2 \right) \\
  &= \log \left(\overline{j(\g,-\overline{z})}\right)  +\overline{H^*_\ci(\g (-\overline{z}))-H^*_\ci(-\overline{z})-
  V_f \left(F_\ci(-\overline{z}) \overline{\langle \g,f\rangle} + | \langle \g,f \rangle|^2/2 \right) }\\
 &= 2\pi i \cdot \rho(\g)+\overline{-2\pi i S^*_\ci(\g)}. \qedhere
\end{align*}
\end{proof}

In the case when $m=0$, proceeding  as in the proof of Proposition \ref{prop:43}, we immediately deduce
that the analogous result holds true for the modular Dedekind symbols $S_\ci$ and hence the differences $\theta_{\ci}$, which is the following assertion.

\begin{prop} \label{43} Let $\Gamma$ be a Fuchsian group, invariant under the action of $\iota$, and with a cusp at $\ci$ and scaling matrix $\si=I$. Then for every $\g$ in $\Gamma$ we have
$$
S_\ci(\iota(\g))=-S_\ci(\g)  -\rho(\g) \quad \text{and}  \quad  \theta_{\ci}(\iota(\g)) = - \theta_{\ci}(\g).
$$
\end{prop}

The above result extends \cite[Lemma 6]{Asai}.

\subsection{Composition with the Fricke involution}

In this section we let $\G = \G_{0}(N)$. Define
\begin{equation*}
    \tau_{N}:=\left(\smallmatrix 0 & -1/\sqrt{N} \\
\sqrt{N} & 0 \endsmallmatrix\right)   \qquad \text{and} \qquad w_N(\g):=\tau_{N} \g \tau_{N}^{-1}.
\end{equation*}
The map $w_{N}$ is an automorphism of $\G_0(N)$ with $w_N^2$ the identity map.
Following \cite[Sect. 6.7]{Iw2} and \cite{AtLeh70}, let $W_N$ be the operator
\begin{equation}\label{wndef}
W_N:f \mapsto  j(\tau_{N},z)^{-2} f(\tau_{N} z)
\end{equation}
mapping $S_2(\G_0(N)) \to S_2(\G_0(N))$.  If  $W_Nf=\epsilon(N,f) \cdot f$
then necessarily $\epsilon(N,f)=\pm 1$.

\begin{lemma} \label{ltop}
Let $f \in S_2(\G_0(N))$ be an eigenfunction of $W_N$ with eigenvalue $\epsilon(N,f)$. Then
\begin{equation*}
    \langle w_N(\g),f\rangle = \epsilon(N,f) \cdot \langle \g,f\rangle.
\end{equation*}
\end{lemma}
\begin{proof}

The action by $\tau_N$ on $\G_{0}(N)$ interchanges the cusps of $\G_0(N)$ at $0$ and  $\ci$, so
then we have
\begin{align*}
F_{\ci}(\tau_{N} z)
&= 2\pi i \int_\ci^{\tau_{N} \ci} f(u) \, du + 2\pi i \int_{\tau_{N} \ci}^{\tau_{N} z} f(u)  \, du\\
&= 2\pi i \int_\ci^{0} f(u)  \, du + 2\pi i \int_{ \ci}^{ z} f(\tau_{N} u) \, d\tau_{N} u\\
&= F_{\ci}(0) + 2\pi i \cdot \epsilon(N,f) \int_{ \ci}^{ z} f( u)  \, du\\
&= F_{\ci}(0) +\epsilon(N,f) F_{\ci}(z).
\end{align*}
To complete the proof,
\begin{multline*}
    \langle w_N(\g),f\rangle = F_{\ci}(w_N(\g) \tau_{N}z) - F_{\ci}(\tau_{N}z) = F_{\ci}( \tau_{N} \g z) - F_{\ci}(\tau_{N}z)\\
    = \epsilon(N,f)(F_{\ci}( \g z) - F_{\ci}(z))= \epsilon(N,f) \langle \g,f\rangle . \qedhere
\end{multline*}
\end{proof}

Let $L(s,f)$ denote the $L-$function associated to the cusp form $f$. Then,
\begin{equation}\label{61}
F_{\ci}(0)=-2\pi i \int_0^{\ci} f(u)  \, du = 2\pi  \int_0^{\ci} f(iy) dy=L(1,f)
\end{equation}
and also $F_{\ci}(0) \in \R$ if $f$ has real Fourier coefficients  at $\ci$ because $L(s,f)-\overline{L(\bar s,f)}=0$ for all $s \in \C$.

\begin{prop} \label{pr61} For all $\g \in \G_0(N)$ we have
$$
S_{ \ci 0}\bigl(w_N(\g)\bigr) = S_{0 \ci}(\g)  \qquad \text{and also} \qquad S^*_{ \ci 0}\bigl(w_N(\g)\bigr) = S^*_{0 \ci}(\g)
$$
when $S^*$ is associated to   $f \in S_2(\G_0(N))$ which is an eigenfunction of $W_N$.
\end{prop}
\begin{proof}
We may take $\si=I$ and $\sigma_{0}=\tau_{N}$ since $w_N(\G_\ci)=\G_0$, $w_N(\G_0)=\G_\ci$. For all integers $m \geq 0$ we have
\begin{align} \notag
E^{m,m}_{\ci}(\tau_{N}z ,s)& =\sum_{\g \in \G_\ci\backslash\G} |\langle\g,f\rangle|^{2m} \Im(\g \tau_{N} z)^s = \sum_{\g \in \G_0\backslash\G} |\langle w_N(\g),f\rangle|^{2m} \Im(w_N(\g) \tau_{N} z)^s\\
& = \sum_{\g \in \G_0\backslash\G} |\epsilon(N,f) \langle \g,f\rangle|^{2m} \Im( \tau_{N} \g  z)^s
 = E^{m,m}_0(z ,s) \label{E infty as E 0}
\end{align}
where the last equality uses $\tau_{N} = -\sigma_{0}^{-1}$. It follows that $B^{(m)}_\ci(\sigma_0 z)=B^{(m)}_0(\si z)$.  With the further equalities
\begin{align*}
    B^{(m)}_\ci(\sigma_0 z) &= - \log y + b^{(m)}_{\ci 0}(0) + 2\Re (H^{(m)}_{\ci 0}(z)),\\
    B^{(m)}_0(\si z) &= - \log y + b^{(m)}_{0 \ci}(0) + 2\Re (H^{(m)}_{0 \ci}(z)),
\end{align*}
we find that $b:=b^{(m)}_{\ci 0}(0) - b^{(m)}_{0 \ci}(0)$ must be real. Consequently
 the real part of the holomorphic function $H^{(m)}_{\ci 0}(z)-H^{(m)}_{0 \ci }(z)+b/2$ is $0$.
Therefore its imaginary part must also be constant and we find
\begin{equation}\label{hhm}
H^{(m)}_{\ci 0}( \g z)-H^{(m)}_{\ci 0}(z) = H^{(m)}_{0 \ci }(\g z)-H^{(m)}_{0 \ci }(z) \quad \text{for all} \quad \g \in \G.
\end{equation}

Assuming $m\geq 1$, equation \eqref{hdiff} gives
\begin{align*}
  \lefteqn{- \log j(\g,z) + V_f\left(F_0( z)\overline{\langle\g,f\rangle} + |\langle\g,f\rangle|^2/2\right) -2\pi i S^*_{0 \ci}(\g)} \hspace{20mm} \\
&= H^{*}_{0 \ci }(\g z)-H^{*}_{0 \ci }(z) \\
&= H^{*}_{\ci 0}( \g z)-H^{*}_{\ci 0}(z)\\
&= H^{*}_{\ci 0}(\sigma_0^{-1} w_N(\g) \sigma_0 z)-H^{*}_{\ci 0}(z) \\
&=  - \log j(\g,z) + V_f \left(F_\ci(\sigma_0 z)\overline{\langle w_N(\g),f\rangle} +  |\langle w_N(\g),f\rangle|^2/2 \right) -2\pi i S^*_{ \ci 0}(w_N(\g)).
\end{align*}
We conclude that
\begin{equation}\label{hhs}
  -2\pi i\bigl(S^*_{0 \ci}(\g) - S^*_{ \ci 0}(w_N(\g)) \bigr)  = \bigl(\epsilon(N,f) F_\ci(\sigma_0 z)-F_0(z)\bigr)\cdot V_f \overline{\langle \g,f\rangle} =0,
\end{equation}
where the last equality in \eqref{hhs} is deduced by letting $z \to 0$. As a result, $S^*_{ \ci 0}\bigl(w_N(\g)\bigr) = S^*_{0 \ci}(\g)$.

The same argument, starting with \eqref{hhm} for $m=0$ and using \eqref{h0} proves $S_{ \ci 0}\bigl(w_N(\g)\bigr) = S_{0 \ci}(\g)$.
\end{proof}

A similar proof to Proposition \ref{pr61}, or combining Propositions \ref{31} and \ref{pr61}, yields the following corollary.
\begin{cor}
For all $\g \in \G_0(N)$ we have $S^*_{ 0 0}\bigl(w_N(\g)\bigr) = S^*_{\ci \ci}(\g)$,
when $S^*$ is associated to   $f \in S_2(\G_0(N))$ which is an eigenfunction of $W_N$.
\end{cor}

\subsection{Composition with Atkin-Lehner involutions}

Let $N$ be a positive integer and choose a positive divisor $v$ of $N$ such that $(v,N/v)=1$.  Then the Atkin-Lehner involution on $\G_0(N)$ is defined by $w_v(\g):= m_v \g m_v^{-1}$ for
\begin{equation}\label{mvch}
m_v:= \frac{1}{\sqrt{v}}\m{av}{b}{Nc}{vd} \quad \text{where} \quad a,b,c,d \in\Z, \ adv^2 - Ncb=v.
\end{equation}
According to \cite{AtLeh70}, the map $w_v$ is an automorphism of $\G_0(N)$ and the operator $W_v: f\mapsto j(m_v, z)^{-2} f(m_v z)$ sends $S_2(\G_0(N))$ to itself with $W_v^2$ the identity. As shown in \cite[Lemma 10]{AtLeh70}, $W_v$ is independent of the choice of $a,$ $b,$ $c,$ $d$ in \eqref{mvch} and hence $W_v$ for $v=N$ is identical to the Fricke involution operator \eqref{wndef} we defined previously.


As described in \cite[Prop. 2.6]{Iw2}, a set of inequivalent cusps for $\G_0(N)$ may be given as
\begin{equation*}
  \bigl\{u/v \ : \ u,v \in \Z, \ 1 \leq v \mid N, \ 1\leq u \leq (v,N/v), \ (u,v)=1 \bigr\}.
\end{equation*}
In \cite[pp. 36-37]{Iw2} it is shown that, when $(v,N/v)=1$, the scaling matrix for the cusp $1/v$ may be taken to be
\begin{equation}\label{s1v}
\sigma_{1/v} = \frac{1}{\sqrt{v_1}} \m{v_1}{b}{N}{dv_1},
\end{equation}
where $v_1:=N/v$ and $b,d\in\Z$ are such that $v_1 d- vb=1$. Clearly \eqref{s1v} takes the form \eqref{mvch} and we write
 the corresponding Atkin-Lehner involution as $w_{1,v}(\g):=\sigma_{1/v} \g \sigma_{1/v}^{-1}$. With this notation we establish the following result.

\begin{prop} \label{prop Atkin Lehner}
Let $v\mid N$ be such that $(v,N/v)=1$ and set $v_1:=N/v$. Suppose $f\in S_2(\G_0(N))$  is an eigenfunction of $W_{v_1}$ with  eigenvalue $\epsilon(v_1,f) = \pm 1$. Then
$$
S_{\ci\ci}^*(\g) = S_{\frac{1}{v}\frac{1}{v}}^*(w_{1,v}(\g)) \quad \text{for all} \quad \g\in\G_0(N).
$$
\end{prop}
\begin{proof}
Proceeding as in Lemma \ref{ltop} we see that
$$
F_{\ci}(\sigma_{1/v}z)= F_{\ci}(1/v) + \epsilon(v_1,f)F_{\ci}(z)
$$
implies  $\langle w_{1,v}(\g),f \rangle = \epsilon(v_1,f) \langle \g,f \rangle$.
Since $w_{1,v}(\G_{\ci}) = \G_{1/v}$,  we find
$$
E_{1/v}^{m,m}(z,s)= \sum_{\g\in\G_{1/v}\backslash \G} |\langle \g, f \rangle|^{2m} \Im (\sigma_{1/v}^{-1}\g z)^s = \sum_{\g\in\G_{\ci}\backslash \G} |\langle w_{1,v}(\g), f \rangle|^{2m} \Im (\g \sigma_{1/v}^{-1} z)^s = E_{\ci}^{m,m}(\sigma_{1/v}^{-1}z,s).
$$
Therefore, $E_{\ci}^{m,m}(z,s)=E_{1/v}^{m,m}(\sigma_{1/v}z,s)$ and $B_{\ci}^{(m)}(z)= B_{1/v}^{(m)}(\sigma_{1/v} z)$. As in the proof of Proposition \ref{pr61}, we deduce that $H_{\frac{1}{v}\frac{1}{v}}^*(\g z) - H_{\frac{1}{v}\frac{1}{v}}^*( z)= H_{\ci \ci}^*(\g z)-H_{\ci \ci}^*(z)$.
Then equation \eqref{hdiff} provides
\begin{align*}
  \lefteqn{- \log j(\g,z) + V_f\left(F_\ci( z)\overline{\langle\g,f\rangle} + |\langle\g,f\rangle|^2/2\right) -2\pi i S^*_{\ci \ci}(\g)} \hspace{20mm} \\
&= H^{*}_{\ci \ci }(\g z)-H^{*}_{\ci \ci }(z) =  H_{\frac{1}{v}\frac{1}{v}}^*(\g z ) - H_{\frac{1}{v}\frac{1}{v}}^*(z) \\
&=H_{\frac{1}{v}\frac{1}{v}}^*(\sigma_{1/v}^{-1} w_{1,v}(\g) \sigma_{1/v} z)- H_{\frac{1}{v}\frac{1}{v}}^*(z) \\
&=  - \log j(\g,z) + V_f \left(F_{1/v}(\sigma_{1/v} z)\overline{\langle w_{1,v}(\g),f\rangle} +  |\langle w_{1,v}(\g),f\rangle|^2/2 \right) -2\pi i S^*_{\frac{1}{v}\frac{1}{v}}(w_{1,v}(\g)).
\end{align*}
To complete the proof observe that
$$
-2\pi i \left( S_{\ci\ci}^*(\g) - S_{\frac{1}{v}\frac{1}{v}}^*(w_{1,v}(\g)) \right) = V_f  \overline{\langle \g,f \rangle} \left(  \epsilon(v_1,f) F_{1/v}(\sigma_{1/v} z) - F_\ci(z) \right)=0,
$$
where the last equality follows by letting $z \to \ci$.
\end{proof}

\subsection{Rationality of higher-order modular Dedekind symbols on genus one congruence groups} \label{rat_thm_prf}

In this section we assume that $N\in \{11, 14, 15, 17, 19, 20, 21, 24, 27, 32, 36, 49\} $ and prove Theorem
\ref{thm: rationality} for the groups $\Gamma_0(N)$. Note that the listed values of $N$ are all the levels for which the surface $\G_0(N)
\backslash \H$ has genus one, see Table 4 of \cite{CuPa03}. Therefore the spaces $S_2(\G_0(N))$ are one dimensional and it is known that they all contain cusp forms with integer Fourier coefficients in their expansions at $\ci$.

Let $\Gamma_0(N)^{\dagger} :=
\G_0(N) \cup \G_0(N)\tau_N$ be the Fricke group. By inspection of the tables presented in \cite{Cu10} (see also formula (4.2) of \cite{CMcKS04}),  one finds that
 the genus of the surface $\Gamma_0(N)^{\dagger} \backslash \H$ is zero for the listed levels of $N$.  Therefore $f\in  S_2(\G_0(N))$ cannot be $\tau_N-$invariant and it follows that $f$ is eigenfunction of $W_N$ with eigenvalue $\epsilon(N,f)=-1$.

As further preparation for the proof of Theorem
\ref{thm: rationality} we require a special case of a formula of Manin, proved in \cite{Ma73}, p. 379 (with $w=0$ and $n=p$).
\begin{prop}\label{Manin}
Let $p$ be prime and $(p,N)=1$. If $f \in S_2(\G_0(N))$ is an eigenfunction of the Hecke operator $T_p$ with $T_p f=a(p) f$ then
\begin{equation}\label{manin f-la}
2\pi i(p+1-a(p))\int_0^{i\ci} f(z)\, dz = \sum_{b=1}^{p-1}\langle (\smallmatrix *
& b \\ * & p \endsmallmatrix ),f \rangle.
\end{equation}
\end{prop}

\begin{proof}[Proof of Theorem \ref{thm: rationality}]
Throughout the proof we will ignore phase factors, since they are integers, and use $\equiv$ to denote the same elements of $\R/\Z$. Let $\ca$ be any cusp of $\G_0(N) \backslash \H$.
Note that $S^*_\ca$ is independent of the associated cusp form $f \not \equiv 0$ in $S_2(\G_0(N))$ since this space is one dimensional and the Kronecker limit defining $S^*_\ca$ is normalized.

An arbitrary element $\g\in\G$ may be represented as a word $\pm I M_1\cdots M_n$, where the matrices $M_i$, $i=1,...,n$ are equal to the group generators, as given in section \ref{basicd}, or their inverses.
 With Corollary \ref{prop:special prop. of theta}, we have $S^*_\ca (-\g) \equiv S^*_\ca( \g)- 1/2$. By repeatedly applying Proposition \ref{cor33} we may reduce the remaining $S^*_\ca(M_1 \cdots M_n)$ to an integer linear combination of $S^*_\ca(M)$ and $S^*_\ca(MN)$ for $M$ and $N$ or their inverses in the set $\{A,B, E_1,..., E_e, P_1, ..., P_c\}$.

Using Proposition \ref{32} we can easily show that
\begin{align*}
    S^*_\ca( M^{-1}N ) & \equiv -S^*_\ca( MN) + 2S^*_\ca( N), \\
    S^*_\ca( M N^{-1} ) & \equiv -S^*_\ca( MN) + 2S^*_\ca( M).
\end{align*}
Hence we may assume that  $M$ and $N$  are in  $\{A,B, E_1,..., E_e, P_1, ..., P_c\}$.
Corollary \ref{par_ell_rat} yields that $S_\ca^*(E_i)$ and $S_\ca^*(P_j)$ are rational for $i=1,...,e$, $j=1,...,c$, while Proposition \ref{32} implies
\begin{equation*}
    S^*_\ca (A^2) \equiv 2S^*_\ca (A), \qquad S^*_\ca (B^2) \equiv 2S^*_\ca (B).
\end{equation*}
It follows that $S^*(\g)$ is a rational linear combination of elements from the set
\begin{equation*}
    \bigl\{1,\ S^*_\ca (A), \ S^*_\ca (B), \ S^*_\ca (AB), \ S^*_\ca (BA)\bigr\}.
\end{equation*}

Corollary \ref{par_ell_rat} also implies that $S_\ca(E_i), S_\ca(P_j) \in\Q$ and hence $\theta_\ca (E_i), \theta_\ca(P_j) \in\Q$ for all $i$ and $j$. Moreover,
an application of Proposition \ref{cor:rpoperties of Psi} to the left relation in \eqref{abab} gives
\begin{align*}
  0=\theta_\ca(I) & = \theta_\ca(ABA^{-1}B^{-1}E_1 E_2 \cdots E_e P_1 P_2 \cdots P_c) \\
   & =\theta_\ca(AB) + \theta_\ca(A^{-1}B^{-1}) + \sum_{i=1}^e \theta_\ca(E_i) +  \sum_{j=1}^c \theta_\ca(P_j).
\end{align*}
Applying Proposition \ref{cor:rpoperties of Psi} again now shows that
\begin{equation} \label{r_n}
\frac{V_f}{\pi}\Im \left(\langle A,f\rangle \overline{\langle B,f\rangle}\right)  \in \Q.
\end{equation}
Therefore, $S_\ca^*(\g)$ is a rational linear combination of elements from the set $ \bigl\{1,\ S^*_\ca (A), \ S^*_\ca (B)\bigr\}$. In order to prove the theorem, it is left to prove that  $S^*_\ca (A)$ and $S^*_\ca (B)$ are rational.

First, we focus on the cusps $\ci$ and $0$ and prove the next lemma.
\begin{lemma} \label{nj}
We have
\begin{equation*} \label{eval}
\frac{V_f}{\pi} \Im \left(\overline{F_{0}( \ci)} \cdot \langle A,f \rangle \right) \in \Q \quad \text{ and    } \quad \frac{V_f}{\pi} \Im \left(\overline{F_{0}( \ci)} \cdot \langle B,f \rangle \right) \in \Q.
\end{equation*}
\end{lemma}
\begin{proof}
We will apply Proposition \ref{Manin}, which is justified, due to the fact that $f$ is an eigenfunction of the Hecke operators $T_p$. Let $p$ be a prime such that $(p,N)=1$ and $p+1-a(p) \neq 0$. Expressing the matrices $\left( \smallmatrix * & b \\ * & p \endsmallmatrix \right)$ in $\G_0(N)$, for $b=1,...,p-1$, in terms of the group generators $\{A,B,E_1,...,E_e, P_1,...,P_c\}$ implies
$$
-(p+1-a(p)) L(1,f) = (p+1-a(p)) F_0(\ci) = u\langle A, f \rangle + v \langle B, f \rangle,
$$
for some  integers $u$, $v$ where we used \eqref{61}. Therefore $F_0(\ci)$ is a rational linear combination of $\s{A}{f}$ and $\s{B}{f}$. This together with \eqref{r_n} completes the proof.
\end{proof}

We now use the operators $\iota$ and $w_N$  to prove that $S^*_\ci (A)$ and $S^*_\ci (B)$ are rational.
Write $\iota(A)$ as a word in the generators. Let $a_1-1$ be the sum of all the exponents of $A$ in this representation and $b_1$ the sum of all the  exponents of $B$. Similarly, write $\iota(B)$ as a word in the generators with $a_2$  the sum of all the  exponents of $A$ and $b_2-1$ the sum of all the  exponents of $B$.

Proposition \ref{prop:43} implies that $S_\ci^*(\iota(A))=-S_\ci^*(A)-\rho(A)$ and similarly for $A$ replaced by $B$. With the
 transformation properties of $S_\ci^*$ and \eqref{r_n} we obtain
\begin{equation} \label{eq. 1}
a_1S_\ci^*(A) +b_1 S_\ci^*(B)=c_1 \quad \text{and} \quad a_2 S_\ci^*(A) + b_2 S_\ci^*(B)=c_2,
\end{equation}
for  $a_1,$ $a_2,$ $b_1,$ $b_2 \in \Z$ and $c_1,$  $c_2 \in \Q$.

Next write $w_N(A)$ as a word in the generators with $a_3+1$  the sum of all the exponents of $A$ in this representation and $b_3$ the sum of all the  exponents of $B$. Similarly, write $w_N(B)$ as a word in the generators with $a_4$  the sum of all the  exponents of $A$ and $b_4+1$ the sum of all the  exponents of $B$.

The relation
\begin{equation} \label{interchange 0 and ci}
S_{\ci}^*(A) - \frac{V_f}{\pi} \Im \left(\overline{F_{0}( \ci)} \cdot \langle A,f \rangle \right) = S_{\ci}^*(w_N(A)) + \omega(\tau_N^{-1}, w_N(A)) - \omega(A,\tau_N^{-1})
\end{equation}
is obtained by application of Corollary \ref{prop relating a and b}, Proposition \ref{pr61}
and Proposition \ref{31}. With the same relation for $B$, and using Lemma \ref{nj}, we find
\begin{equation*} \label{eq. 1x}
a_3S_\ci^*(A) +b_3 S_\ci^*(B)=c_3 \quad \text{and} \quad a_4 S_\ci^*(A) + b_4 S_\ci^*(B)=c_4,
\end{equation*}
for  $a_3,$ $a_4,$ $b_3,$ $b_4 \in \Z$ and $c_3,$  $c_4 \in \Q$.

Define
\begin{equation*}
  J_{ij}:=\m{a_i}{b_i}{a_j}{b_j} \quad \text{and} \quad M:=\m{\langle A,f \rangle}{\overline{\langle A,f \rangle}}{\langle B,f \rangle}{\overline{\langle B,f \rangle}}.
\end{equation*}
From part (iii) of Lemma \ref{iotalem} we know  $\overline{\langle A,f \rangle}= \langle\iota(A),f \rangle$ (using that $f$ has real Fourier coefficients). Also by Lemma \ref{ltop}, $-\langle A,f \rangle = \langle w_N(A),f \rangle$ since $\epsilon(N,f)=-1$.  The same two equalities are true for $B$. Hence,
\begin{equation*}
  J_{13}M= \m{a_1}{b_1}{a_3}{b_3} \m{\langle A,f \rangle}{\overline{\langle A,f \rangle}}{\langle B,f \rangle}{\overline{\langle B,f \rangle}}
  = 2\m{\Re\langle A,f \rangle}{\Re\langle A,f \rangle}{-\langle A,f \rangle}{-\overline{\langle A,f \rangle}}
\end{equation*}
for example, and so
\begin{equation*}
  \det(J_{13})\det(M) = 8i \cdot \Re\langle A,f \rangle \cdot \Im\langle A,f \rangle.
\end{equation*}
Similarly,
\begin{align*}
  \det(J_{14})\det(M) & = 8i \cdot \Re\langle A,f \rangle \cdot \Im\langle B,f \rangle, \\
  \det(J_{23})\det(M) & = 8i \cdot \Re\langle B,f \rangle \cdot \Im\langle A,f \rangle, \\
  \det(J_{24})\det(M) & = 8i \cdot \Re\langle B,f \rangle \cdot \Im\langle B,f \rangle.
\end{align*}

Suppose that none of $J_{13},$ $J_{14},$ $J_{23},$ $J_{24}$ is invertible. Then the right sides above are all zero. If $\Re\langle A,f \rangle \neq 0$ or $\Re\langle B,f \rangle \neq 0$ then it follows that $\Im\langle A,f \rangle =0$ and $\Im\langle A,f \rangle=0$. But with \eqref{zagier} this contradicts that $f$ is not identically zero. Hence $\Re\langle A,f \rangle = 0$ and $\Re\langle B,f \rangle = 0$, but this also cannot happen for nonzero $f$. We conclude that at least one   $J_{ij}$  is invertible and therefore $ S_\ci^*(A)$ and $ S_\ci^*(B)$ are both rational since
\begin{equation*}
  J_{ij} \begin{pmatrix} S_\ci^*(A)  \\ S_\ci^*(B)  \end{pmatrix} = \begin{pmatrix} c_i  \\ c_j  \end{pmatrix}.
\end{equation*}

We have established that $S_\ci^*(\g) \in \Q$ for all $\g \in \G_0(N)$. We may extend this to $S_{1/v}^*(\g)$
using
Proposition \ref{prop Atkin Lehner} when the cusp $1/v$ has $v\mid N$ and $(v,N/v)=1$.
For $\sigma_{1/v}$ given by \eqref{s1v}, let $\g_1 := \sigma_{1/v}^{-1}\g\sigma_{1/v} \in\G_0(N)$ so that $\g=w_{1,v}(\g_1)$. Hence Proposition \ref{prop Atkin Lehner}, together with Proposition \ref{31}, show that
$$
S_{\frac{1}{v}}^*(\g) \equiv S_{\frac{1}{v}\frac{1}{v}}^*(\g) =  S_{\frac{1}{v}\frac{1}{v}}^*(w_{1,v}(\g_1))= S_{\ci\ci}^*(\g_1)\in\Q.
$$
It now follows from Proposition \ref{indep} that $S_\ca$ and $S^*_\ca$ are rational for all cusps $\ca$ that are $\G$-equivalent to $1/v$ for the above values of $v$.
This includes the cusp $0$ which is equivalent to $1/1$.
The proof is complete.
\end{proof}

Theorem \ref{thm: rationality} omits some of the cusps in the groups $\Gamma_0(N)$ for $N\in \{ 20,  24, 27, 32, 36, 49\} $ when the level is not square-free.

\section{Examples}

In these final sections we present computations of the  modular Dedekind symbols for  $\G_0(11)$, the smallest level congruence group of genus one, and  $\G_0(37)^+$, the smallest level moonshine type group of genus one.

\subsection{Evaluations for $\G_0(11)$}
The space $S_2(\G_0(N))$ is zero for $N \leq 10$ and one-dimensional for $N=11$, so we take $\G=\G_0(11)$ as our first case of interest.
The quotient space
$X_{0}(11):=\G_0(11)\backslash \H$ has genus $g=1$ and volume $4\pi$. Let
\begin{equation}\label{sstt}
S = \left(\smallmatrix 0 & -1 \\ 1 & 0 \endsmallmatrix\right)
\quad
\text{and}
\quad
T = \left(\smallmatrix 1 & 1 \\ 0 & 1 \endsmallmatrix\right).
\end{equation}
Then a set of generators for $\G_0(11)$ is given by the  elements
\begin{align*}
A&= \left(\smallmatrix -7 & -1 \\ 22 & 3 \endsmallmatrix\right)= -I S T^3 S T^{-7} S,
& P_0&= \left(\smallmatrix 1 & 0 \\ -11 & 1 \endsmallmatrix\right)= -I S T^{11} S,\\
B&= \left(\smallmatrix 4 & 1 \\ -33 & -8 \endsmallmatrix\right)= S T^8 S T^{-4} S,
& P_\ci&= \left(\smallmatrix 1 & 1 \\ 0 & 1 \endsmallmatrix\right)= T
\end{align*}
along with $-I$.
They satisfy the relation $ABA^{-1}B^{-1}P_0P_\ci=I$, and one
can construct a fundamental domain for the action of $\G_0(11)$ on $\H$ which has two inequivalent cusps:
one  at $\ci$ and the other at $0$.
We may take the corresponding scaling matrices to be  $\sigma_\ci=I$ and
$\sigma_0=\tau_{11}$.

Since the group $\G=\G_0(11)$ is fixed in this section, we will just use the notation $S_\ca$ for the modular Dedekind symbol and the notation $S^{*}_{\ca}$ for the higher-order modular Dedekind symbol, both associated to the cusp $\ca$.
Employing \eqref{xi N for comgruence} and \eqref{xi N for comgruence2}, one can directly compute that
\begin{equation*}\label{52}
S_{\ci}(-I)=-\frac 12, \ \  S_{\ci}(A)=-\frac{2}{5}, \ \ S_{\ci}(B)=\frac{2}{5}, \ \ S_{\ci}(P_0)=0, \ \ S_{\ci}(P_\ci)=1.
\end{equation*}
It follows from part (i) of Theorem \ref{main1} that $10 S_{\ci}(\g) \in \Z$ for all $\g \in \G_0(11)$.

The space  $S_2(\G_0(11))$ is generated by   the newform $f(z) = \eta(z)^2\eta(11z)^2$. The Fourier
expansion of $f$
 follows directly from Ramanujan's expression for the Dedekind eta function, namely
$$
\eta(z)=\sum_{n=1}^\infty \chi(n)e^{2\pi i z n^2/24} \quad \text{where} \quad \chi(n):=\begin{cases} 1 &\text{ \ for\ }n\equiv \pm 1 \bmod 12
\\
-1&\text{ \ for\ }n\equiv \pm 5 \bmod 12
\\
0&\text{ \ otherwise,\ }
\end{cases}
$$
which in turn may be proved using the Jacobi triple product formula; see \cite[p. 29]{Bu}.
Therefore
$$
f(z)=q_z-2q_z^2-q_z^3+2q_z^4+q_z^5+2q_z^6-2q_z^7+ \cdots
$$
has integer Fourier coefficients and is an eigenfunction of the operator $W_{11}$.

\begin{prop}\label{lin-comb} For every $\g$ in $\G_0(11)$ we have
$$
S^*_\ca(\g)\in \frac{1}{2} \Z + S^*_\ca(A)\Z + S^*_\ca(B)\Z.
$$
\end{prop}
\begin{proof}
Proceeding  as in the proof of Theorem \ref{thm: rationality} we deduce that  $S_\ca^*(\g)$ is integer linear combination of $1/2$ and elements $S^*_\ca(M)$ and $S^*_\ca(MN)$ for $M$ and $N$ in the set $\{P_\ci,P_0,A,B\}$.

For the parabolic elements $P_\ci$ and $P_0$, we may use Proposition \ref{prop35} to see that
\begin{equation} \label{pp00}
    S^*_\ca (P_\ci) =0, \qquad S^*_\ca (P_0) = 0.
\end{equation}
Therefore, with Proposition \ref{32},
\begin{equation*}
    S^*_\ca (PM)\equiv S^*_\ca (MP) \equiv S^*_\ca (P) + S^*_\ca (M) \equiv S^*_\ca (M)
\end{equation*}
for $P=P_\ci$ or $P_0$. With all of the above, we have shown that $S^*_\ca( \g)$ is an integer linear combination of elements from the set
\begin{equation} \label{sset}
    \bigl\{1/2, \ S^*_\ca (A), \ S^*_\ca (B), \ S^*_\ca (A^2), \ S^*_\ca (B^2), \ S^*_\ca (AB), \ S^*_\ca (BA)\bigr\}.
\end{equation}
Now, following the proof of Theorem \ref{thm: rationality}, by applying $\theta_\ca$ to $ABA^{-1}B^{-1}P_0P_\ci$ we have that
\begin{equation} \label{1/2}
    \frac{V_f}{2\pi}\Im (\langle A,f\rangle \overline{\langle B,f\rangle}) = - \frac 12 (\theta_\ca(P_0)+\theta_\ca(P_\ci)) = \frac 12.
\end{equation}
Reducing the elements in \eqref{sset} with Proposition \ref{32} we have
\begin{equation*}
    S^*_\ca (A^2) \equiv 2S^*_\ca (A), \qquad S^*_\ca (B^2) \equiv 2S^*_\ca (B)
\end{equation*}
easily, and using \eqref{1/2},
\begin{equation*}
    S^*_\ca (AB) \equiv S^*_\ca (BA) \equiv S^*_\ca (A)+  S^*_\ca (B) +1/2.
\end{equation*}
This finishes the proof of Theorem \ref{lin-comb}.
\end{proof}

In the case that $\ca =\ci$ we can be even more explicit, using our work in section \ref{symd} to show the next result.

\begin{theorem} \label{thm: comp A and B}
We have
\begin{equation} \label{ooo}
    S^*_\ci (P_\ci) \equiv 0,  \quad S^*_\ci (P_0) \equiv 0,  \quad S^*_\ci (A) \equiv \frac{9}{10},  \quad S^*_\ci (B) \equiv \frac{1}{10}
\end{equation}
where the equivalences are in $\R/\Z$. Also, for all $\g \in \G_0(11)$,
\begin{equation}\label{10z}
    10S^*_\ci (\g) \in \Z.
\end{equation}
\end{theorem}
\begin{proof}
Applying Proposition \ref{43}, Corollary \ref{prop:special prop. of theta} and Proposition \ref{32} we get
\begin{align}
    S^*_\ci (A) & = -S^*_\ci (\iota(A)) = -S^*_\ci (P_0^{-1} B P_0) \notag\\
    &   \equiv  -S^*_\ci (P_0^{-1}) -  S^*_\ci (B) - S^*_\ci (P_0) \equiv  -S^*_\ci (B). \label{ttt0}
\end{align}
Note also that
\begin{equation}\label{55}
\overline{\langle A,f\rangle} = \langle \iota(A),f\rangle = \langle P_0^{-1} B P_0 ,f\rangle = \langle  B  ,f\rangle.
\end{equation}

We next assemble two further identities that will allow us to  evaluate $ S^*_\ci (A)$.
 Applying \eqref{manin f-la} for $p=2$ to our situation we find
$$
-5L(1,f)=(a(2)-3)L(1,f)=\langle (\smallmatrix *
& 1 \\ * & 2 \endsmallmatrix ),f \rangle = \langle (\smallmatrix 6
& 1 \\ 11 & 2 \endsmallmatrix ),f \rangle = \langle -P_0^{-1}BA,f \rangle = \langle A,f \rangle+ \langle B,f \rangle.
$$
With (\ref{55}) this means that
\begin{equation}\label{l1f}
L(1,f)=-\frac{2}5 \Re \langle A,f \rangle.
\end{equation}

Now, simplifying \eqref{interchange 0 and ci} with $w_{11}(A)=P_0^{-1}A^{-1}P_\ci^{-1}$ and computing the phase factors shows
\begin{equation*}\label{idyr}
    S^*_\ci (A) = \frac{V_f}{2\pi} \Im \left(\overline{F_{0}( \ci)} \cdot \langle A,f \rangle \right).
\end{equation*}
Then, with (\ref{61}), $F_{0}( \ci)$ is given by \eqref{l1f} and so
\begin{align*}
    S^*_\ci (A) & \equiv -\frac{2}5  \cdot \frac{V_f}{2\pi} \Re \langle A,f \rangle \Im  \langle A,f \rangle  \\
    & \equiv -\frac{1}5  \cdot \frac{V_f}{2\pi} \Im\left( \langle A,f \rangle^2 \right) \\
    & \equiv -\frac{1}5  \cdot \frac{V_f}{2\pi} \Im\left( \langle A,f \rangle \overline{\langle B,f \rangle}\right) \equiv -\frac{1}5 \cdot\frac{1}2
\end{align*}
using \eqref{55} and \eqref{1/2}. We have shown that $S^*_\ci (A) \equiv 9/10$ and so \eqref{ooo} is now a consequence of \eqref{ttt0} and \eqref{pp00}. Finally, \eqref{10z} follows from \eqref{ooo} and Proposition \ref{lin-comb}.
\end{proof}

It is an easy exercise with  Theorem \ref{main2},  part (ii) to show the same result for the  modular Dedekind symbol $S^*_0$ associated to the other cusp:

\begin{cor}
We have
\begin{equation*}
    S^*_0 (P_\ci) \equiv 0,  \quad S^*_0 (P_0) \equiv 0,  \quad S^*_0 (A) \equiv \frac{1}{10},  \quad S^*_0 (B) \equiv \frac{9}{10}
\end{equation*}
where the equivalences are in $\R/\Z$. For all $\g \in \G_0(11)$ we have $10S^*_0 (\g) \in \Z$.
\end{cor}

\subsection{Computations for  $\G_0(37)^+$}

The smallest square-free $N$ such that $\G_0(N)^+$ possesses a holomorphic cusp form $f$ of weight
two is $N=37$.
The compactified quotient $X_{0}(37)^{+}:=\G_0(37)^+\setminus \H$ is isomorphic to the genus one curve $y^2 =4x^3-4x+1$ over $\C$
in the sense that one has an isomorphism between $\G\setminus \H \cup \{\ci\}$ and the algebraic curve $y^2 =4x^3-4x+1$
sending $\ci$ to $0$ and whose pull-back of the canonical differential $dx/dy$ is $-2\pi i f(z)dz$. The $q$-expansion of $f$ is given by
$$
f(z)=q_z-2q_z^2-3q_z^3+2q_z^4-2q_z^5+6q_z^6-q_z^7+6q_z^9+ \cdots.
$$
Its Petersson norm on $\G_0(37)^+$ is $\|f\|^2=\omega_1\omega_2 /(4\pi^2 i)$ where  $\omega_1$ and $\omega_2$ are the real and complex periods of the curve
$y^2 =4x^3-4x+1$ with approximate  values $\omega_1\approx2.993458644$ and $\omega_2\approx2.451389381 i$.
We refer the reader to \cite{Za85} for these and further results regarding the algebraic and arithmetic geometry of
$X_{0}(37)^{+}$.

A fundamental domain for $X_{0}(37)^{+}$ can be constructed using results from \cite{Cu10} and the group $\G_0(37)^+$ may be presented with the
 generators
\begin{gather*}
   A =\ns\ms{148}{-89}{185}{-111}, \quad B=\ms{20}{-13}{37}{-24}, \quad  E_1 =-\tau_{37}=\ns\ms{0}{1}{-37}{0},  \\
  E_2=\ms{-6}{1}{-37}{6}, \quad
  E_3 =\ms{-11}{3}{-37}{10}, \quad E_4=\ns\ms{37}{-19}{74}{-37}, \quad P_{\ci} =\ms{1}{-1}{0}{1},
\end{gather*}
and the relations
\begin{equation}\label{37relns}
  ABA^{-1}B^{-1}E_4E_3E_2E_1P_{\ci}=I, \quad E_4^2=E_2^2 = E_1^2=-I, \quad E_3^3=I.
\end{equation}
The Riemann surface $X_{0}(37)^{+}$ has genus 1 with one cusp, four elliptic points,  and the volume of the surface is $19\pi/3$.

Let $S_\ci$ denote the modular Dedekind symbol for $\G_0(37)^+$. A computation with Proposition \ref{prop: S p +}  shows that
\begin{gather}\label{37elpr}
S_\ci(P_{\infty})= -\frac{19}{12}, \quad S_\ci(E_1)= S_\ci(E_2)=\frac{1}{4},  \quad  S_\ci(E_3)=\frac{1}{3},  \quad   S_\ci(E_4)=-\frac{1}{4},\\
S_\ci(A)= \frac{1}{6}, \quad S_\ci(B)=-\frac{7}{12}, \notag
\end{gather}
where the results for the parabolic and elliptic elements in \eqref{37elpr} may also be derived from section \ref{ellpr}.
Now, Theorem \ref{main1}, part (i) yields that $12 S_\ci(\g) \in\Z$ for all $\g \in\G_0(37)^+$.

\begin{prop} Let $S^*_\ci$ be the higher-order modular Dedekind symbol for $\G_0(37)^+$. Then
for all $\g \in \G_0(37)^+$,
\begin{equation*}
  S^*_\ci(\g) \in \frac{1}{24}\Z + S^*_\ci(B)\Z.
\end{equation*}
\end{prop}
\begin{proof}
The values of  $S_{\infty}^*$ on parabolic and elliptic elements are found using Propositions
\ref{prop35} and  \ref{prop36}:
$$
S_\ci^*(P_{\infty})=0, \quad S_\ci^*(E_1) =S_\ci^*(E_2)=\frac{1}{4}, \quad  S_\ci^*(E_3)=\frac{1}{3},  \quad  S_\ci^*(E_4)=-\frac{1}{4}.
$$
Proceeding  as in the proof of Theorem \ref{thm: rationality}, we see that  $S^*_\ci( \g)$ is always an integer linear combination of elements from the set $\bigl\{1/12,\ S^*_\ci (A), \ S^*_\ci (B), \ S^*_\ci (AB), \ S^*_\ci (BA)\bigr\}$. Applying $\theta_\ci$ to the left identity in \eqref{37relns} shows that
\begin{equation*}
  \frac{V_f}{\pi}\Im \left(\langle A,f\rangle \overline{\langle B,f\rangle}\right)=
  -\theta_\ci(E_4)-\theta_\ci(E_3)-\theta_\ci(E_2)-\theta_\ci(E_1)-\theta_\ci(P_\ci) = -\frac{19}{12}.
\end{equation*}
Therefore $S^*_\ci( \g)$ is in $\frac 1{12}\Z+ S^*_\ci(A)\Z + S^*_\ci(B)\Z$.

We use the operator $\iota$ next to evaluate $S_\ci^*(A)$. To see how the isomorphism $\iota$ acts on the group generators we first note that $\iota(E_1)=-E_1$ and $\iota(P_\ci)=P_\ci^{-1}$. A numerical computation of the modular symbols $\s{A}{f}$, $\s{B}{f}$ indicates that $\iota(B)$ is a word in the generators where the sum of the exponents of $A$ is $0$ and the sum of the exponents of $B$ is $-1$. Writing all the generators in terms of $S$ and $T$ in \eqref{sstt}, leads to the identity $\iota(B)=-P_\ci E_4 B^{-1}E_4 P_\ci^{-1}$. Applying $\iota$ to both sides of this identity also shows $\iota(E_4)=-P_\ci E_4  P_\ci^{-1}$. Applying $\iota$ to the left identity in \eqref{37relns} now shows after some experimentation that
\begin{equation}\label{tyu}
\iota(A)= -P_{\ci} E_3^2 E_4 BAB^{-1}E_4P_\ci^{-1}.
\end{equation}
 Next apply  $\theta_\ci$ to \eqref{tyu} and use Proposition \ref{43} and Proposition \ref{cor:rpoperties of Psi}  to deduce that
$$
-\theta_\ci(A)= \theta_\ci(\iota(A))= \theta_\ci(BAB^{-1})= \theta_\ci(A)-\frac{V_f}{\pi}\Im \left(\langle A,f\rangle \overline{\langle B,f\rangle}\right).
$$
Therefore, $\theta_\ci(A) = -19/24$, and hence
$
S_\ci^*(A)= \theta_\ci(A) + S_\ci(A) =  -5/8.
$
\end{proof}

We have not determined the value of $S_\ci^*(B)$ and whether it is rational. In the following section we conclude by showing how $S_\ci^*$ for $\G_0(p)^+$ can be related to the higher-order modular Dedekind symbol $S_\ci^*$ for $\G_0(p)$. In this way, understanding $S_\ci^*$ for $\G_0(37)$ would allow the evaluation of $S_\ci^*(B)$ above.

\subsection{A relation between higher-order modular Dedekind symbols on $\G_0(p)$ and $\G_0(p)^+$}

Let $p$ be a prime. In this section we express the higher-order modular Dedekind symbols on $\G_0(p)^+$ as an arithmetic mean of two higher-order modular Dedekind symbols on $\G_0(p)$.  Additional notation  is needed  to indicate the two different groups under consideration: we use the standard notation for functions on the group $\G_0(p)$, and will add $+$ in the index to denote the corresponding functions on $\G_0(p)^+$.

Take $f$ to be a weight two cusp form on $\G_0(p)^+$ (and hence a weight two cusp form on $\G_0(p)$).
Since $f(\tau_p z ) = j(\tau_p, z)^2 f(z)$, we obviously have $W_pf=f$.

If $\g\in\G_0(p)^+$ then $\g^2\in \G_0(p)$ and $S_{\infty,+}^*(\g)=\frac{1}{2}(S_{\infty,+}^{*}(\g^2) - \omega(\g,\g))$, where $\omega(\g,\g)$ is explicitly evaluated in Lemma \ref{lemma omega}. So the symbol $S_{\infty,+}^*$ is determined by its values on $\G_0(p)$.

\begin{prop} \label{prop: S to S+}
For all $\g\in\G_0(p)$
\begin{equation} \label{S as an arith mean}
S_{\ci,+}^*(\g)=\frac{1}{2}\bigl(S_{\ci}^*(\g) + S_{\ci 0}^*(\g) \bigr) = S_{\ci}^*(\g) +\frac{1}{2}\left(\omega(\tau_p^{-1},\g) - \omega(\tau_p^{-1} \g\tau_p, \tau_p^{-1})\right).
\end{equation}
\end{prop}
\begin{proof}
 With \eqref{disjoint}, the fundamental domain $\mathcal{F}_p$ of $\G_0(p)$ is a disjoint union of $\mathcal{F}_{p,+}$ and $\mathcal{F}_{p,+}\tau_p$. Therefore, $V_{\G_0(p)} = 2V_{\G_0(p)^+}$ and $\|f\|^2=2\|f\|_{+}^2$, where  $\|f\|$ indicates the Petersson norm on $\G_0(p)$ and $\|f\|_+$  the norm on $\G_0(p)^+$. This immediately yields that $A_{m,+}=2A_m$. Moreover, $\tau_p$ is elliptic, hence $\langle \g\tau_p,f \rangle = \langle \g,f\rangle$, for all $\g\in\G_0(p)^+$. Therefore, for all $m\geq 1$, $z\in\H$ and $\Re(s)>1$ we have
\begin{align*}
  E_{\ci,+}^{m,m} (z,s) & =\sum_{\g\in\G_{\ci} \backslash \G_0(p)^+} |\langle\g,f \rangle|^{2m} (\Im \g z) ^s \\
   & =\sum_{\g\in\G_{\ci} \backslash \G_0(p)} |\langle\g,f \rangle|^{2m} (\Im \g z) ^s + \sum_{\g\in\G_{\ci} \backslash \G_0(p)} |\langle\g,f \rangle|^{2m} (\Im \g \tau_pz) ^s
\end{align*}
and hence
$$
E_{\ci,+}^{m,m} (z,s) = E_{\ci}^{m,m} (z,s) + E_{\ci}^{m,m} (\tau_p z,s)= E_{\ci}^{m,m} (z,s) + E_{0}^{m,m} (z,s),
$$
as in equation \eqref{E infty as E 0}.
The definition of $B_{\ci,+}^{(m)}(z)$ yields the equation
\begin{align*}
B_{\ci,+}^{(m)}(z) &= \lim_{s \to 1} \left[ \frac{1}{2} \left( (s-1)^m \frac{E_\ci^{m,m}(z,s)}{A_m} - \frac{1}{s-1} \right) +  \frac{1}{2} \left( (s-1)^m \frac{E_0^{m,m}(z,s)}{A_m} - \frac{1}{s-1} \right)\right]\\
&= \frac{1}{2}\left(B_\ci^{(m)}(z) + B_0^{(m)}(z)\right),
\end{align*}
and we may relate the expansions
\begin{align*}
B_{\ci,+}^{(m)} (z)&= -\log y + b_{\ci\ci,+}^{(m)}(0) + 2 \Re(H_{\ci\ci,+}^{(m)}(z))\\
&=- \log y + \frac{1}{2}\left(b_{\ci\ci}^{(m)}(0) + b_{0\ci}^{(m)}(0)\right) + \Re \left(H_{\ci\ci}^{(m)}(z) + H_{0\ci}^{(m)}(z)\right).
\end{align*}
Using the $*$ notation \eqref{h_ind_m}, it follows that $2H_{\ci\ci,+}^{*}(z) - (H_{\ci\ci}^{*}(z) + H_{0\ci}^{*}(z))$ is an imaginary constant. Therefore
$$
2\bigl(H_{\ci\ci,+}^{*}(\g z) - H_{\ci\ci,+}^{*}(z)\bigr)=  \bigl(H_{\ci\ci}^{*}(\g z) - H_{\ci\ci}^{*}(z)\bigr)  + \bigl(H_{0\ci}^{*}(\g z) - H_{0\ci}^{*}(z)\bigr),
$$
for all $\g \in\G_0(p)$.
Equation \eqref{hdiff} together with the fact that $V_{f,+}=V_f$ implies
$$
V_{f}F_\ci(z) \overline{\langle \g,f\rangle} -2\pi i S_{\ci\ci,+}^*(\g) = \frac{1}{2}V_f\overline{\langle \g,f\rangle} \bigl(F_{\ci}(z) + F_0(z)\bigr)- \pi i \bigl(S_{\ci\ci}^*(\g) + S_{\ci 0}^{*}(\g)\bigr)
$$
and letting $z \to \ci$, we find
$$
-2\pi i \left(S_{\ci\ci,+}^*(\g) - \frac{1}{2}(S_{\ci\ci}^*(\g) + S_{\ci 0}^{*}(\g)) \right)=\frac{1}{2}V_f\overline{\langle \g,f\rangle}F_0(\ci).
$$
Also
$$
F_0(\ci) = 2\pi i \int_0^{\ci} f(w)dw= 2\pi i \int_{\tau_p (\ci)}^{\ci} f(w)dw = - \langle \tau_p, f \rangle=0,
$$
so we immediately deduce the first equality in \eqref{S as an arith mean}.
The second follows from Proposition \ref{31}.
\end{proof}

A straightforward computation of the phase factors in \eqref{S as an arith mean} using Proposition \ref{cases} yields:

\begin{cor} Let $\g=\left( \smallmatrix a & b \\ c & d \endsmallmatrix \right) \in\G_0(p)$. Then,
\begin{equation}\label{sabcg}
S_{\ci,+}^*(\g)=
               \begin{cases}
                 S_{\ci}^*(\g) + 1/2, &  \quad \text{if } c\geq 0, \,\, a \leq 0 \text{  and  } b > 0; \\
                 S_{\ci}^*(\g) -1/2, & \quad \text{if }  c < 0, \,\, a \leq 0 \text{  and  } b\leq  0; \\
                 S_{\ci}^*(\g), & \quad\text{otherwise.}
               \end{cases}
\end{equation}
\end{cor}

\begin{remark}\rm
It is clear that the proof of Proposition \ref{prop: S to S+} may be extended from $p$ to a positive integer $N$ and that the modular Dedekind symbols  $S_{\ci, \dagger}^*(\g)$ on $\G_0(N)^{\dagger}:=\G_0(N)\cup \G_0(N)\tau_N$
are related to  the symbols $S_{\ci}^*(\g)$ on $\G_0(N)$ by the same formula \eqref{sabcg} for all $\g \in \G_0(N)$.
\end{remark}

{\small


\begin{align*}
& \text{Jay Jorgenson} & &  \text{Cormac O'Sullivan} & & \text{Lejla Smajlovi\'c} \\
& \text{Department of Mathematics} & & \text{Department of
Mathematics} & & \text{Department of
Mathematics} \\
& \text{City College of New York} & & \text{The CUNY Graduate Center} & & \text{University of Sarajevo} \\
& \text{Convent Avenue at 138th Street} & & \text{365 Fifth Avenue} & & \text{Zmaja od Bosne 35, 71 000 Sarajevo} \\
& \text{New York, NY 10031} & & \text{New York, NY 10016} & & \text{Bosnia and Herzegovina} \\
& \text{e-mail: jjorgenson@mindspring.com} & & \text{e-mail: cosullivan@gc.cuny.edu} & & \text{e-mail: lejlas@pmf.unsa.ba}
\end{align*}

}

\end{document}